\documentclass{article}
\usepackage{graphicx} 
\usepackage{mathrsfs}
\usepackage{geometry}
\usepackage{titletoc}
\usepackage{amsmath}
\allowdisplaybreaks[4]
\usepackage{amssymb}
\usepackage{mathtools}
\usepackage{enumerate}
\usepackage{amsthm}
\usepackage{caption}
\usepackage{subcaption}
\usepackage[backref]{hyperref} 
\hypersetup{
hidelinks
}

\usepackage[table]{xcolor} 
\usepackage[all]{xy}
\usepackage{tikz}

\newcommand\cD{{\mathcal D}}

\newcommand\cF{{\mathcal F}}

\newcommand\cL{{\mathcal L}}

\newcommand\cO{{\mathcal O}}
\newcommand\cP{{\mathcal P}}

\newcommand\cX{{\mathcal X}}

\newcommand{\ka}{\mathfrak{a}}
\newcommand{\kb}{\mathfrak{b}}

\newcommand{\codim}{{\rm codim}}

\newcommand{\ch}{\mathrm{ch}}
\newcommand{\Td}{\mathrm{Td}}

\newcommand{\Vol}{\mathrm{Vol}}
\newcommand{\wvol}{\widetilde{\mathrm{Vol}}}
\newcommand{\ord}{\mathrm{ord}}

\newcommand{\Bc}{\mathrm{Bc}}
\newcommand{\wBc}{\widetilde{\Bc}}
\newcommand{\Val}{\mathrm{Val}}
\newcommand{\lct}{\mathrm{lct}}
\newcommand{\Hom}{\mathrm{Hom}}

\newcommand{\ssa}{{\mathfrak a}}
\newcommand{\ssb}{{\mathfrak b}}

\newcommand{\CC}{\mathbb {C}}

\newcommand{\FF}{{\mathbb F}}

\newcommand{\NN}{{\mathbb N}}
\newcommand{\PP}{{\mathbb P}}
\newcommand{\QQ}{{\mathbb Q}}
\newcommand{\RR}{{\mathbb R}}
\newcommand{\ZZ}{{\mathbb Z}}

\DeclareMathOperator{\Int}{int}
\DeclareMathOperator{\Aut}{Aut}
\DeclareMathOperator{\Ver}{Ver}

\DeclareMathOperator{\Proj}{Proj}
\DeclareMathOperator{\Ric}{Ric}
\DeclareMathOperator{\argmin}{argmin}
\DeclareMathOperator{\Fut}{Fut}

\theoremstyle{plain}
\newtheorem{theorem}{Theorem}[section]
\newtheorem{proposition}[theorem]{Proposition}

\newtheorem{lemma}[theorem]{Lemma}

\newtheorem{claim}[theorem]{Claim}
\newtheorem{corollary}[theorem]{Corollary}

\newtheorem{conjecture}[theorem]{Conjecture}

\newtheorem{problem}[theorem]{Problem}

\theoremstyle{definition}
\newtheorem{definition}[theorem]{Definition}
\newtheorem{remark}[theorem]{Remark}
\newtheorem{example}[theorem]{Example}

\def\K{K\"ahler }
\def\KE{K\"ahler--Einstein }
\newcommand{\beq}{\begin{equation}}
\newcommand{\eeq}{\end{equation}}
\newcommand{\bpf}{\begin{proof}}
\newcommand{\epf}{\end{proof}}
\newcommand{\bdefn}{\begin{definition}}
\newcommand{\edefn}{\end{definition}}
\newcommand{\bremark}{\begin{remark}}
\newcommand{\eremark}{\end{remark}}
\newcommand{\bconj}{\begin{conjecture}}
\newcommand{\econj}{\end{conjecture}}
\newcommand{\bcor}{\begin{corollary}}
\newcommand{\ecor}{\end{corollary}}
\newcommand{\blem}{\begin{lemma}}
\newcommand{\elem}{\end{lemma}}
\newcommand{\bclaim}{\begin{claim}}
\newcommand{\eclaim}{\end{claim}}
\newcommand{\bprob}{\begin{problem}}
\newcommand{\eprob}{\end{problem}}
\newcommand{\bprop}{\begin{proposition}}
\newcommand{\eprop}{\end{proposition}}
\newcommand{\bthm}{\begin{theorem}}
\newcommand{\ethm}{\end{theorem}}
\def\lb#1{\label{#1}}
\def\ra{\rightarrow}

\def\q{\quad}

\font\itten=cmti10

\title{Asymptotics of quantized barycenters of lattice polytopes \\
with applications to algebraic geometry
}
\author{Chenzi Jin, Yanir A. Rubinstein
\\
\\
{\itten with an appendix by
 Yaxiong Liu}
 \\ \\
 }

\date{June 2024}

\begin{document}

\maketitle

\begin{abstract}
This article addresses a combinatorial problem with
applications to algebraic geometry.
To a convex lattice polytope $P$ and each of its integer dilations
$kP$
one may associate the barycenter
of its lattice points. This sequence of $k$-quantized
barycenters converge to the
(classical) barycenter of the polytope considered as a convex body. 
A basic question arises: is there a complete asymptotic expansion
for this sequence? If so, what are its terms?
This article initiates the study of this question. 
First, we establish the existence of such an expansion
as well as determine the first two terms.
Second, for Delzant lattice polytopes we use toric algebra
to determine all terms using mixed volumes of virtual
rooftop polytopes, or alternatively in terms of higher
Donaldson--Futaki invariants. Third, for reflexive polytopes we show
the quantized barycenters are colinear to first order,
and actually colinear in the case of polygons. 
The proofs use Ehrhart theory, convexity arguments, and toric algebra.
As applications we derive the complete asymptotic expansion
of the Fujita--Odaka stability thresholds $\delta_k$ on arbitrary
polarizations on (possibly singular) toric varieties.
In fact, we show they are rational functions of $k$ for sufficiently large $k$. 
This gives the first
general result on Tian's stabilization problem for $\delta_k$-invariants for (possibly singular) toric Fanos: 
$\delta_k$ stabilize in $k$ if and only if they are all equal to $1$, and when smooth if and only if asymptotically Chow semistable (a condition stronger
than existence of a K\"ahler--Einstein metric).
We also relate
the asymptotic expansions to higher Donaldson--Futaki invariants
of test configurations motivated by Ehrhart theory, and unify in passing
previous results of Donaldson, Futaki, Ono, Sano, and Rubinstein--Tian--Zhang
on existence of canonical K\"ahler metrics, obstructions, 
and stability thresholds.

\end{abstract}

\tableofcontents

\section{Introduction}
\label{Introduction}

\subsection{A combinatorial problem}

This article initiates the study of a combinatorial problem:
{\it what is the asymptotic behavior of the quantized (or discrete) barycenters of lattice polytopes?}

To state this more precisely, let 
$$M\cong\ZZ^n$$ 
denote a lattice and $P\subset M\otimes_\ZZ\RR\cong\RR^n$
denote a convex lattice polytope, i.e., the vertex set
satisfies $\Ver P\subset M$. 
The {\it $k$-th quantized barycenter} (or $k$-th discrete barycenter) of $P$
is defined as the average of the lattice points in $kP$ divided by $k$:
$$
\Bc_k(P):=\frac1{k|kP\cap M|}\sum_{u\in kP\cap M}u, \q k\in\NN.
$$
The barycenter of $P$ 
is the first moment of the uniform probability measure on $P$,
considered as a convex body in 
$M_\RR:=M\otimes_\ZZ\RR\cong\RR^n$, or
$$
\Bc(P):=\frac1{|P|}\int_Pxdx\in M_\RR.
$$
The empirical measures $\frac1{|kP\cap M|}\sum_{u\in P\cap k^{-1}M}\delta_u$
converge weakly to $1_Pdx/|P|$ so
$$
\Bc(P)=
\lim_{k\ra\infty}\Bc_k(P)\in M_\RR.
$$
Does the sequence $\{\Bc_k(P)\}_{k\in\NN}$ encode
other canonical quantities associated to $P$ beside $\Bc(P)$? We pose the following:
\bprob
\lb{ExpansionProb}
Is there an asymptotic expansion 
$$
\Bc_k(P)=\sum_{i=0}^\infty a_i(P)k^{-i}?
$$
If so, what are the  coefficients $a_i(P)$?
\eprob

Problem \ref{ExpansionProb} is of purely combinatorial interest within Ehrhart theory.
Thus, the bulk of this article (Sections \ref{Bc_k lattice}, \ref{Bc_k reflexive}, and \ref{delta_k-invariant for toric del Pezzo surfaces}) 
can be read as a purely enumerative combinatorics paper, and 
to some extent also
Section \ref{Bc_k smooth} that is concerned with algebraic combinatorics. 
Sections \ref{delta_k section}--\ref{tc section} and the Appendix
are concerned with the algebraic geometric applications and can be read independently by the reader interested only in those (or skipped by the combinatorics-oriented reader).

\subsection{Relation to algebraic-geometry and stability thresholds}

At the same time, this article can be considered as the second in a series where we initiate the study of asymptotics
of algebro-geometric invariants arising from Kodaira embeddings using 
vector spaces of holomorphic sections of
powers of a line bundle $H^0(X,L^k)$. The asymptotic parameter is
the power  $k$ of the line bundle, where $1/k$ can be thought
of as Planck's constant in geometric quantization. 
The over-arching geometric motivation for the series initiated in \cite{JR1} is 
to obtain complete asymptotic expansions
for these invariants. This is of course motivated
by the foundational result of Catlin and Zelditch 
on the complete asymptotics
of the Bergman kernel\cite{Cat99,Zel98}, but the setting is quite different: the invariants
we are interested in are Tian's $\alpha_k$- and Fujita--Odaka's $\delta_k$-invariants \cite{Tian90,FO18}
and they are not directly defined in terms of the Bergman kernel. Rather, they 
are more related to singularities of divisors in the linear series $|kL|$ 
on $X$. These invariants play a fundamental role in both the minimal model program in algebraic geometry and in the study of canonical K\"ahler metrics in differential geometry. 

\begin{problem}
\lb{deltaProb}
Is there an asymptotic expansion 
$$\delta_k=\sum_{i=0}^\infty a_i(X,L)k^{-i}?$$
If so, what are the coefficients $a_i(X,L)$?
\end{problem}

By a theorem of Blum--Jonsson $\lim_k\delta_k=\delta$ \cite[Theorem 4.4]{BJ20},
with $\delta$ the basis log canonical threshold of Fujita--Odaka
\cite{FO18}.
Thus, if such an expansion exists then $a_0(X,L)=\delta$. 

Another motivation is {\it Tian's stabilization problem} which asks whether
such invariants ``stabilize", i.e., become eventually constant in $k$.
Tian asked this for his $\alpha_k$-invariants in the smooth Fano setting \cite[Question 1]{Tian90}, and later for all polarizations on possibly singular varieties \cite[Conjecture 5.3]{Tian12}, and one can extend his problem to the $\delta_k$-invariants.

\begin{problem}
\lb{TianProb}
Is $\delta_k=\delta$ for all sufficiently large $k$?
\end{problem}

One can consider our Problem \ref{deltaProb} as a refinement of Problem \ref{TianProb}.

We completely solve Problem \ref{TianProb} in the (possibly singular) toric Fano setting (Corollary \ref{asymChow}).
This shows there is a remarkable dichotomy between the $\alpha_k$-invariants and the $\delta_k$-invariants. The first paper in the series resolved Tian's stabilization problem for $\alpha_k$-invariants by showing that in the toric Fano setting $\alpha_k=\alpha$ for all $k$. Here we show that $\delta_k$-invariants behave very differently and {\it rarely stabilize,} 
but that if they do stabilize then they also satisfy $\delta_k=\delta$ for {\it all} $k\in\NN$, and when Fano also $\delta_k=\delta=1$.
This precisely
suggests that most likely there is a complete asymptotic expansion as conjectured
in Problem \ref{deltaProb}, and indeed we prove that is the case.

\medskip
Problems \ref{deltaProb} and \ref{TianProb} are intimately related to 
Problem \ref{ExpansionProb} as we explain next. 

\medskip

The relation between stability thresholds and barycenters has a rich and long history.
Let $\Omega\in H^2(M,\RR)$ be a \K class.
In 1983 Futaki introduced the eponymous invariant $F_\Omega(v)$ that associates to 
each holomorphic
vector field $v$ on $X$ a number; he proved that $F_\Omega$ must vanish on all holomorphic
vector fields in order for a constant scalar curvature \K metric to exist in $\Omega$
\cite{Fut83}.
In the case $X$ is toric it suffices to consider the $n$
vector fields $t_1,\ldots,t_n$ 
generated by the complex torus $(\CC^\star)^n$ acting on $X$. A 
beautiful theorem of Mabuchi states that (when suitably normalized)
$(F(t_1),\ldots,
F(t_n))\in\RR^n$ is the barycenter $\Bc(P_{X,\Omega})$ of the Delzant polytope 
$P_{X,\Omega}$ associated to 
$(X,\Omega)$ \cite[Theorem 5.3]{Mab87}.
Thus, Futaki's obstruction can be rephrased geometrically as 
$\Bc(P_{X,\Omega})=0$. 
Some thirty years later, Blum--Jonsson showed that Mabuchi's theorem
can be rephrased in terms of the Fujita--Odaka $\delta$ invariant \cite[Corollary 7.16]{BJ20}:
$$
\delta^{-1}=\max_{i=1,\ldots,d}\big[\langle\Bc(P_{X,\Omega}),v_i\rangle+b_i\big],
$$
where $\{v_1,\ldots, v_d\}$ are the primitive generators of the one-dimensional cones of the fan determining $X$
and $(b_1,\ldots, b_d)\in\RR^d$ determine the class $\Omega$, so that
$P_{X,\Omega}=\bigcap_{i=1}^d\left\{u\in M_\RR\,:\,
\left\langle u,v_i\right\rangle+b_i\ge0\right\}$. In particular, in the Fano
case $\Omega=-K_X$, $b_1=\ldots=b_d=1$, so $\delta\le1$ with equality if and only if $\Bc(P)=0$.
It was then realized by Rubinstein--Tian--Zhang that Blum--Jonsson's formula in
the Fano case extends also to quantized barycenters \cite[Corollary 7.1]{RTZ21},
\beq
\lb{RTZEq}
\delta_k^{-1}=\max_{i\in\{1,\ldots,d\}}\big[\langle\Bc_k(P_{X,-K_X}),v_i\rangle+1\big],
\eeq
and that similarly $\Bc_k(P_{X,-K_X})$ is essentially
a quantized Futaki character that obstructs the existence
of $k$-th balanced metrics \cite[Theorem 2.3]{RTZ21}.
The formula \eqref{RTZEq} is the link, at the heart of the present article, between quantized barycenters and stability thresholds. 
In an appendix it is shown that \eqref{RTZEq} extends to all polarizations \eqref{toric delta_k formula}.

Thus, the study of asymptotics of $\Bc_k(P)$ 
for lattice polytopes is essentially equivalent to
the study of the asymptotics of $\delta_k$ for (possibly singular) toric varieties.
We employ this link in both directions in this article. 
For general $P$ we derive a formula for $\Bc_k(P)$ which
answers the first part of Problem \ref{ExpansionProb} as well as 
determines $a_1(P)$; we then use this to obtain  
corresponding formula and asymptotics for $\delta_k$ for general (possibly 
very singular) toric varieties.
Restricting to toric manifolds we then use (as well as add some
entries to) the well-developed
dictionary between Riemann--Roch theory and Ehrhart theory to answer
Problem \ref{ExpansionProb} completely for {\it Delzant lattice polytopes} $P$, 
determining all $a_i(P)$ in terms of mixed volumes of virtual polytopes. 
Thirdly, restricting to Gorenstein Fano varieties amounts to working with
the sub-class of reflexive polytopes for which we refine the general expansion
and show that $\Bc_k(P)$ are colinear, to first order, with $\Bc(P)$.
This is not true for general $P$.
Finally, for toric del Pezzos we show how the combinatorial and algebraic
approach give two completely different proofs of the same  
formula for $\Bc_k(P)$ or $\delta_k$. In this dimension
$\Bc_k(P)$ are actually colinear with $\Bc(P)$. The beauty of this special case is
that either convexity arguments or Ehrhart theory allows to replace the otherwise mysterious Ehrhart polynomial
expressions by completely explicit formulas involving the lattice-normalized measure or intersection numbers.

In a forthcoming article, the third in the series, 
we unify the study of Tian's and Fujita--Odaka's invariants \cite{JR3}.

\section{Results}

Our overarching goal is to make this article accessible to both combinatorists
and algebraic geometers. As part of that, we carefully separate the sections dealing with
each.
Sections \ref{Bc_k lattice}--\ref{delta_k-invariant for toric del Pezzo surfaces} detail our combinatorial results on Problem \ref{ExpansionProb}.
Section \ref{delta_k section} describes our results on Problems \ref{deltaProb} and \ref{TianProb}. Section \ref{tc section} revisits Section \ref{Bc_k smooth}
in terms of test configurations and is relevant to both audiences.

\subsection{Results for general lattice polytopes}

To state our most general result on Problem \ref{ExpansionProb} we set-up some
basic terminology from Ehrhart theory.

A real Euclidean space $V$ of dimension $n$ and an $\RR$-basis 
$\{v_1,\ldots,v_n\}$ 
of $V$ define a lattice 
$$
\ZZ^n\cong M\equiv M(v_1,\ldots,v_n):=\bigg\{\sum_{i=1}^n a_iv_i\,:\, a=(a_1,\ldots,a_n)\in\ZZ^n\bigg\}
\subset V\cong\RR^n,
$$
and $\{v_1,\ldots,v_n\}$ is called a basis of $M$, and 
$M\otimes_\ZZ\RR=V\cong\RR^n$. Any other basis
of $M$ differs by an $n$-by-$n$ matrix with integer entries and determinant
$\pm1$ \cite[Theorem 21.1]{Gruber}.
The determinant of $M$ is the Euclidean volume of the unit (or fundamental) 
parallelotope
$\{\sum_{i=1}^na_iv_i\,:\, a\in[0,1)^n\}\subset M\otimes_\ZZ\RR$, denoted
$$
\det M=|\det[v_1\ldots v_n]|,
$$
independent of the choice of basis.

A convex lattice polytope $P\subset M_\RR$ is the convex
hull of a finite collection of lattice points, assumed
to be $n$-dimensional.

\begin{definition}
\label{normalized vol}
    Let $L$ be a $n$-dimensional lattice and $L_H$ an affine sublattice obtained by intersecting $L$ with a hyperplane $H$.

    \noindent $\bullet\;$ Let $\Vol$ denote the standard Lebesgue $(n-1)$-dimensional measure on $H$. Define the \emph{lattice-normalized measure on $H$} as the measure with respect to the affine sublattice $L_H$, i.e.,
    $$
    \wvol:=\frac{1}{\Vol(C)}\Vol,
    $$
    where $C$ is a fundamental parallelotope of the affine sublattice $L_H$. That is, $C$ is the parallelotope spaned by a basis of $L_H$. Note that while fundamental parallelotopes are not unique, their volumes are the same \cite[p. 358]{Gruber}.

    \noindent $\bullet\;$ Let $P$ be a lattice polytope. The {\it lattice-normalized measure on $\partial P$} is defined to be the measure whose restriction to any facet of $P$ is the normalized volume on that facet. Let $\wBc(\partial P)$ denote the barycenter of $\partial P$ with respect to this measure, i.e.,
    $$
        \wBc(\partial P)=\frac{1}{\wvol(\partial P)}\int_{\partial P}u\cdot d\wvol
    $$
\end{definition}

\bigskip
See Examples \ref{example1 of normalized vol} and \ref{example2 of normalized vol}
for some explicit calculation of the lattice-normalized measure.
This measure is important in Ehrhart theory, as it governs the first non-trivial
coefficient of the {\it Ehrhart polynomial}, a classical object recalled in the next result:

\begin{theorem}{\rm \cite{Ehr67a,Ehr67b}, \cite[Theorem 19.1]{Gruber},\cite[Theorem 12.2, Exercise 12.5]{MS05},\cite[Corollary 5.5]{BR07}}
\label{Ehrhart polynomial}
    Let $M$ be a lattice of dimension $n$. For any lattice polytope $P$, there is a polynomial
    $$
        E_P(k)=\sum_{i=0}^na_ik^i
    $$
    such that
    $$
        E_P(k)=\#(kP\cap M)
    $$
    for any $k\in\NN$. In particular,
    $$\begin{aligned}
        a_n&=\Vol(P),&a_{n-1}&=\frac{1}{2}\wvol(\partial P),&a_0&=1.
    \end{aligned}$$
\end{theorem}

\medskip

Next, embed $M$ as a sublattice $M\times\{0\}$ of $M\times\ZZ$. 
For $c=(c_1,\ldots,c_n)\in M$ such that 
$$
P+c\subset\RR_+^n\times\{0\},
$$
and for each $i\in\{1,\ldots,n\}$, define
the convex lattice polytope of one dimension higher (``rooftop polytope"),
$$
P(c;i):=
\Big\{(u,h)\,:\, u=(u_1,\ldots,u_n)\in 
P+c,\; h\in[0,u_i] 
\Big\}
\subset M_\RR\times\RR.
$$
Denote 
\beq
\lb{BccompEq}
\Bc_k(P)=(\Bc_{k,1}(P),\ldots,\Bc_{k,n}(P))\in M_\RR\cong\RR^n.
\eeq

\medskip
Our first main result answers the first part of Problem \ref{ExpansionProb},
as well as determines the first non-trivial coefficient.

\bthm
\lb{FirstMainThm}
Let $M\cong\ZZ^n$ be a lattice and $P\subset M_\RR$ be an $n$-dimensional 
convex lattice polytope. For $c\in M$ such that $P+c\subset\RR_+^n$,
    \begin{equation}\label{translation}
        \Bc_{k,i}
        =\frac{E_{P(c;i)}(k)-E_P(k)}{kE_P(k)}-c
                =\frac{E_{P(c;i)}(k)-E_{P\times[0,c_i]}(k)}{kE_P(k)}.
    \end{equation}
    In particular, $\Bc_k(P)=\sum_{i=0}^\infty a_i(P)k^{-i}$ for some $a_i$
    that depend only on $P$. Moreover,
    $$
    a_0(P)=\Bc(P),\q
    a_1(P)=\frac{\wvol(\partial P)}{2\Vol(P)}(\wBc(\partial P)-\Bc(P)),
    $$
    where $\wvol(\partial P)$ is the mass of the lattice-normalized
    measure of $\partial P$, and $\wBc(\partial P)$ is
    its barycenter (see Definition \ref{normalized vol}). 
Finally, if $\Bc_k(P)$ is constant for $n+1$ values of $k$ then $a_i(P)=0$
for all $i\ge1$ and $\Bc_k(P)=\Bc(P)$ for all $k\in\NN$.
\ethm

Theorem  \ref{FirstMainThm} implies that if 
$\Bc(P)\not=\Bc_k(P)$ for {\it some} $k\in\NN$ then 
the same is true for {\it all sufficiently large} $k$.
In this context, it raises the following interesting problem:
\bprob
\lb{stabProb}
Assume that $P$ is a lattice polytope satisfying 
$\Bc(P)\not=\Bc_k(P)$ for some $k\in\NN$.
Does there exist such a $P$ for which $\{\Bc_{i}(P)\}_{i\in I}$
is a singleton for some subset $I\subset \NN$ with $|I|\ge2$?
\eprob

Theorem \ref{FirstMainThm}
says that if such a $P$ exists then $|I|\le n$. In Remark
\ref{rmk equivalence of Bc=0} we resolve Problem \ref{stabProb}
in dimension $n=2$.

\subsection{Results for Delzant lattice polytopes}\label{Delzant subsec}

In this part, we assume that $P$ belongs to the sub-class of Delzant polytopes.

\begin{definition}\label{defi: DelzantPolytope}
    A convex lattice polytope $P$ in $\mathbb{R}^n$ is {\it Delzant} if:
    \begin{enumerate}
        \item there are exactly $n$ edges meeting at each vertex $p$ (i.e., $P$ is {\it simple});
        \item the edges meeting at a vertex $p$ satisfy that each edge is of the form $p+t u_i$ with $u_i\in\mathbb{Z}^n$, $i=1,\dots,n$;
        \item the $u_1,\dots,u_n$ in (2) can be chosen to be a basis of $\mathbb{Z}^n$.
    \end{enumerate}
\end{definition}

Such polytopes are an important sub-class of lattice polytopes since there is a one-to-one correspondence between Delzant lattice polytopes and 
polarized toric manifolds. We collect in one statement 
that correspondence and in fact a more general one including
the singular and reflexive cases.

\begin{theorem}
\lb{toric-corresp}
{\rm \cite[Theorem 28.2]{dasilva}, \cite{Delzant}, \cite[\S1.5]{Ful93}, \cite[Proposition 2.2.23]{Bat94}}
    There is a one-to-one correspondence between lattice polytopes and polarized toric varieties. Specifically, let
    $$
        P=\bigcap_{i=1}^d\left\{u\in M_\RR\,:\,\left\langle u,v_i\right\rangle\geq-b_i\right\}
    $$
    be a lattice polytope, where each $v_i\in N:=\Hom(M,\ZZ)$ is chosen to be primitive, i.e., not a multiple of another lattice point, and $b_i\in\ZZ$. It determines a fan where $\{v_i\}_{i=1}^d$ is the set of primitive generators of the one-dimensional cones. This fan then determines a toric variety $X$, on which the irreducible toric divisors $\{D_i\}_{i=1}^d$ are in one-to-one correspondence with $\{v_i\}_{i=1}^d$. The divisor
    $$
        D:=\sum_{i=1}^db_iD_i
    $$
    is ample and Cartier. Let $L$ be the corresponding line bundle. Then the polarized variety $(X,L)$ is the one in correspondence with $P$.
    
    The variety $X$ is smooth if and only if $P$ is Delzant (Definition \ref{defi: DelzantPolytope}).

    The variety $X$ is Gorenstein Fano (i.e., $-K_X$ is Cartier and ample) with $L=-K_X$ if and only if $P$ is reflexive.
\end{theorem}

To generalize this correspondence for divisors that are not necessarily ample, we use the notion of virtual polytopes \cite[Corollary 2]{PK93}.

\begin{definition}\label{virtual def}
\noindent
$\bullet\;$ Let $\cP$ denote the monoid of (possibly degenerate, i.e., of lower dimension) polytopes in $M_\RR$. Recall that a monoid is a set equipped with an associative multiplication with identity. Let $\cP^*$ denote the Grothendieck group of $\cP$, i.e., the group of equivalent classes of formal differences $P_1-P_2$ where $P_1-P_2\sim P_1'-P_2'$ if and only if $P_1+P_2'=P_1'+P_2$. The elements in $\cP^*$ are called {\it virtual polytopes}.

\noindent
$\bullet\;$     The notion of {\it mixed volumes} (including volumes) can be extended by multi-linearity to virtual polytopes. For $P_1,\ldots,P_k\in \cP^*$ and 
$m_1,\ldots,m_k\in\NN$ with $m_1+\ldots+m_k=n$, set
    $$
        V\left(P_1,m_1;\ldots;P_k,m_k\right):=V\left(P_1,\ldots,P_1,\ldots,P_k,\ldots,P_k\right),
    $$
    where each $P_i$ appears $m_i$ times.

    \noindent
$\bullet\;$ Let $X$ be a smooth toric variety (toric manifold). Let 
$$\{v_i\}_{i=1}^d$$ denote the primitive generators of the rays in the associated fan. Let $$D_i$$ denote the toric divisor corresponding to $v_i$ \cite[\S3.3]{Ful93}. For any ample toric divisor
    $$
        D=\sum_{i=1}^db_iD_i
    $$
    where $b_i\in\ZZ$, define the {\it polytope associated to $(X,D)$},
    $$
        P_D:=\bigcap_{i=1}^d\left\{u\in M_\RR\,:\,\left\langle u,v_i\right\rangle\geq-b_i\right\}.
    $$
    In general, since any divisor can be written as the difference of two ample divisors \cite[Example 1.2.10]{Laz04}, for any divisor $D$ there is an associated virtual polytope $P_D\in \cP^*$.
\end{definition}

There are efficient numerical algorithms to compute mixed volumes,
see, e.g., \cite{TV07,GMRV}.

\bigskip

Let $P$ be a Delzant lattice polytope, and let 
$(X,L)$ be the associated polarized toric manifold given by Theorem \ref{toric-corresp},
where
$$
    L=\sum_{i=1}^db_iD_i.
$$
Fix $v\in N$. Pick 
\beq\lb{qEq}
\hbox{$q\in\ZZ$ such that $\langle\cdot,v\rangle+q>0$ on $P$ (possible as $P$ is compact).}
\eeq
Consider the lattice polytope of one dimension higher,
\begin{align}
\lb{PvqEq}
    P_{v,q}&:=\left\{\left(u,h\right)\in\RR^n\times\RR\,:\,u\in P,\ 0\leq h\leq\left\langle u,v\right\rangle+q\right\}\\
    &\phantom{:}=\left\{\left(u,h\right)\in\RR^n\times\RR\,:\,\left\langle u,v_i\right\rangle\geq-b_i, i\in\{1,\ldots,d\},\ 0\leq h\leq\left\langle u,v\right\rangle+q\right\}.
\end{align}
It corresponds to a polarized toric manifold of dimension $n+1$
\begin{equation}\label{cXcL}
    \left(\overline{\cX},\overline{\cL}\right)
\end{equation}
(see Figures \ref{fan} and \ref{test configuration}). This polytope has $d+2$ facets. Under the correspondence
of Theorem \ref{toric-corresp},
the associated 
primitive generators of the rays in the fan are 
$$
(v_1,0),\ldots,(v_d,0),(0,1),(v,-1)\in N\times\ZZ,
$$
and we denote
\begin{equation}\label{cD def}
    \cD_1,\ldots,\cD_{d},\cD_{d+1},\cD_{d+2}
\end{equation}
the corresponding torus-invariant divisors. 
Rewrite \eqref{PvqEq},
$$
    P_{v,q}=\left\{\left(u,h\right)\in\RR^n\times\RR\,:\,\left\langle\left(u,h\right),\left(v_i,0\right)\right\rangle\geq-b_i,\ \left\langle\left(u,h\right),\left(0,1\right)\right\rangle\geq0,\ \left\langle\left(u,h\right),\left(v,-1\right)\right\rangle\geq-q\right\},
$$
By Definition \ref{virtual def}, 
$$
P_{v,q}=P_\cD,
$$
where
$$
    \cD:=\sum_{i=1}^db_i\cD_i+q\cD_{d+2},
$$
and
$$
\overline{\cL}\sim\cD.
$$
Define a virtual polytope
$$
    P_v':=P_{\sum\limits_{i=1}^db_i\cD_i}=P_{v,q}-qP_{\cD_{d+2}}.
$$

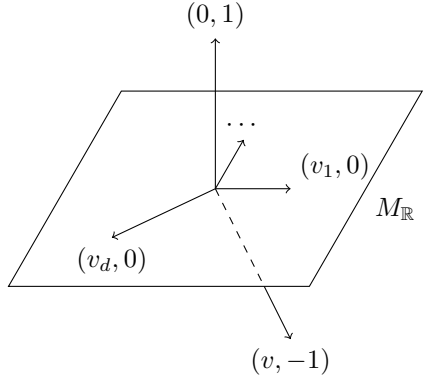
\begin{figure}
    \centering
    \begin{tikzpicture}
        \draw(-2-3/4,{-3*sqrt(3)/4})--(2-3/4,{-3*sqrt(3)/4})--(2+3/4,{3*sqrt(3)/4})node[midway,below right]{$M_\RR$}--(-2+3/4,{3*sqrt(3)/4})--cycle;
        \draw[->](0,0)--(1,0)node[above right]{$(v_1,0)$};
        \draw[->](0,0)--(3/8,{3*sqrt(3)/8})node[above]{$\cdots$};
        \draw[->](0,0)--(-1-3/8,{-3*sqrt(3)/8})node[below]{$(v_d,0)$};
        \draw[->](0,0)--(0,2)node[above]{$(0,1)$};
        \draw[dashed](0,0)--({3*sqrt(3)/8},{-3*sqrt(3)/4});
        \draw[->]({3*sqrt(3)/8},{-3*sqrt(3)/4})--(1,-2)node[below]{$(v,-1)$};
    \end{tikzpicture}
    \caption{\small The fan of the new toric manifold $\overline{\cX}$. Note that $X$ corresponds to the subfan consisting of the cones in $M_\RR\subseteq M_\RR\times\RR$. The cones above $M_\RR$ are spanned by the cones in $M_\RR$ and the ray generated by $(0,1)$, while the cones below $M_\RR$ are spanned by the cones in $M_\RR$ and the ray generated by $(v,-1)$.}\label{fan}
\end{figure}

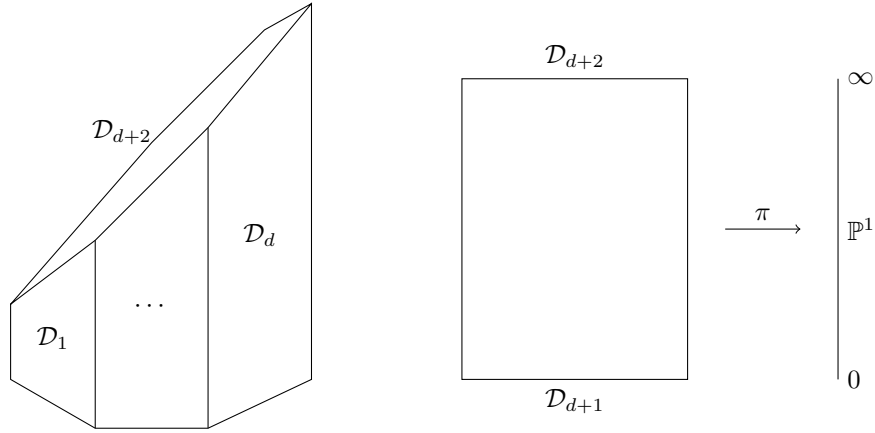
\begin{figure}
    \centering
    \begin{tikzpicture}
        \draw(-2,1)--(-2,0)--(-.5-3/8,{-3*sqrt(3)/8})--(1-3/8,{-3*sqrt(3)/8})--(2,0)--(2,5)--(1-3/8,{4-3*sqrt(3)/8})--(-.5-3/8,{2.5-3*sqrt(3)/8})--cycle--(-.5+3/8,{2.5+3*sqrt(3)/8})--(1+3/8,{4+3*sqrt(3)/8})--(2,5);
        \draw(-.5-3/8,{-3*sqrt(3)/8})--(-.5-3/8,{2.5-3*sqrt(3)/8});
        \draw(1-3/8,{-3*sqrt(3)/8})--(1-3/8,{4-3*sqrt(3)/8});
        \draw(-23/16,{7/8-3*sqrt(3)/16})node{$\cD_1$};
        \draw(-1/8,{13/8-3*sqrt(3)/8})node{$\cdots$};
        \draw(21/16,{9/4-3*sqrt(3)/16})node{$\cD_d$};
        \draw(0,3)node[above left]{$\cD_{d+2}$};
        \draw(4,0)--(7,0)node[midway,below]{$\cD_{d+1}$}--(7,4)--(4,4)node[midway,above]{$\cD_{d+2}$}--cycle;
        \draw[->](7.5,2)--(8.5,2)node[midway,above]{$\pi$};
        \draw(9,0)node[right]{$0$}--(9,4)node[midway,right]{$\PP^1$}node[right]{$\infty$};
    \end{tikzpicture}
    \caption{The projection $\pi:\overline{\cX}\to\PP^1$.}\label{test configuration}
\end{figure}

\bigskip 
Our second main result is the complete resolution of Problem 
\ref{ExpansionProb} when $P$ is Delzant:

\bigskip 
\begin{theorem}\label{smooth theorem}
    Let $P\subseteq M\otimes\RR$ be a Delzant lattice polytope. Then for any $v\in N$,
    $$
        \left\langle\Bc_k\left(P\right),v\right\rangle=\frac{\sum\limits_{j=0}^nc'_{j+1}k^j}{\sum\limits_{j=0}^na_jk^j}=\sum_{j=0}^\infty k^{-j}\sum_{i=0}^jc'_{n+1-j+i}\left(\frac{\delta_{i0}}{a_n}+\sum_{\ell=1}^i\frac{\left(-1\right)^\ell}{a_n^{\ell+1}}\sum_{\sum\limits_{m=1}^\ell i_m=i-\ell}\prod_{m=1}^\ell a_{n-i_m-1}\right),
    $$
    where
    \begin{align*}
        a_j&=\sum_{\sum\limits_{i=1}^d\ell_i=n-j}\frac{n!B_{\ell_1}\cdots B_{\ell_d}}{j!\ell_1!\cdots\ell_d!}V\left(P,j;P_{D_1},\ell_1;\ldots;P_{D_d},\ell_d\right),\\
        c'_j&=\sum_{\sum\limits_{i=1}^d\ell_i=n+1-j}\frac{\left(n+1\right)!B_{\ell_1}\cdots B_{\ell_d}}{j!\ell_1!\cdots\ell_d!}V\left(P_v',j;P_{\cD_1},\ell_i;\ldots;P_{\cD_d},\ell_d\right).
    \end{align*}
    In particular,
    \begin{align*}
        a_n&=\Vol\left(P\right),\\
        a_{n-1}&=\frac{n}{2}V\left(P,n-1;P_{-K_X},1\right),\\
        a_{n-2}&=\frac{n\left(n-1\right)}{12}\left(V\left(P,n-2;P_{-K_X},2\right)+\sum\limits_{i<j}V\left(P,n-2;D_i,1;D_j,1\right)\right),\\
        c'_{n+1}&=\Vol\left(P_v'\right),\\
        c'_n&=\frac{n+1}{2}V\left(P_v',n;P_{-K_{\overline{\cX}/\PP^1}},1\right),\\
        c'_{n-1}&=\frac{\left(n+1\right)n}{12}\left(V\left(P_v',n-1;P_{-K_{\overline{\cX}/\PP^1}},2\right)+\sum\limits_{1\leq i<j\leq d}V\left(P_v',n-1;P_{\cD_i},1;P_{\cD_j},1\right)\right),
    \end{align*}
    where
    \begin{align*}
        P_{-K_X}=P_{\sum\limits_{i=1}^dD_i},
        \q\q
        P_{-K_{\overline{\cX}/\PP^1}}=P_{\sum\limits_{i=1}^d\cD_i},
    \end{align*}
    and
    \begin{multline*}
        \left\langle\Bc_k\left(P\right),v\right\rangle=\frac{c'_{n+1}}{a_n}+\left(\frac{c'_n}{a_n}-\frac{c'_{n+1}a_{n-1}}{a_n^2}\right)\frac{1}{k}\\
        +\left(\frac{c'_{n-1}}{a_n}-\frac{c'_na_{n-1}+c'_{n+1}a_{n-2}}{a_n^2}+\frac{c'_{n+1}a_{n-1}^2}{a_n^3}\right)\frac{1}{k^2}+O\left(\frac{1}{k^3}\right).
    \end{multline*}
\end{theorem}

\subsection{Results for reflexive polytopes}

Another subclass of lattice polytopes is that of reflexive polytopes.

\begin{definition}\label{refdef}
{\rm\cite[Definition 2.3.12]{CLS11}}
    A lattice polytope is called {\it reflexive} if it contains the origin, and each of its facet is given by $\langle u,v\rangle=1$ where $v\in N$ is some primitive vector.
\end{definition}

\begin{remark}
    We are not aware of a relationship between being reflexive and being Delzant. For instance, the square with vertices $(\pm1,0)$ and $(0,\pm1)$ is reflexive but not Delzant, and the square with vertices $(\pm2,\pm2)$ and $(\pm2,\mp 2)$ is Delzant but not reflexive.
\end{remark}

Our third main result is the specialization of Theorem \ref{FirstMainThm}
to the reflexive setting. It is quite remarkable that in this
setting the {\it quantized barycenters are, to first order, 
colinear}. Furthermore, in the planar case, they are in fact {\it colinear}, 
and have a completely explicit expression.

\begin{theorem}
\lb{ReflexiveThm}
    Let $P$ be a reflexive polytope. Then $a_1(P)=\Bc(P)/2$ 
    (recall Problem \ref{ExpansionProb}), i.e.,
    $$
        \Bc_k(P)=\Bc(P)\left(1+\frac{1}{2k}\right)+O(k^{-2}).
    $$
    Moreover, if $\dim P=2$, then
    $$
        \Bc_k(P)=\frac{(k+1)(2k+1)\wvol(\partial P)}{4+2k(k+1)\wvol(\partial P)}\Bc(P).
    $$
\end{theorem}

We give two completely different proofs of the two-dimensional part
of Theorem \ref{ReflexiveThm}.

\subsection{Applications to asymptotics of \texorpdfstring{$\delta_k$}{delta k}-invariants}

Theorem \ref{FirstMainThm} together with Corollary \ref{toric delta_k} answer the first part of Problem \ref{deltaProb} 
as well as determine the first non-trivial coefficient for a general
polarized (possibly singular) toric variety.

\begin{corollary}
\label{MainAsympThm}
Let $X$ be a (possibly singular) toric variety given by a fan $\Delta$ in a lattice $N\simeq\ZZ^n$. 
Let $v_1,\ldots,v_d\in N$ denote the primitive generators of rays in $\Delta$.
Let $L$ be an ample toric line bundle with the associated convex polytope (recall Theorem \ref{toric-corresp})
\begin{align*}
    P=\{x\in M_\mathbb{R}~:~ \langle x,v_i\rangle\geq -b_i,\ i=1,\ldots,d\}
\end{align*}
for some $b_i\in\RR$, where $M_\RR:=\Hom(N,\ZZ)\otimes\RR$. Let $I$ be the set of indices $i$ such that $v_i$ and $b_i$ compute $\delta$ $($see \textup{(\ref{toric delta})}$)$. Then  there is
a complete asymptotic expansion for $\delta_k$
as in Problem \ref{deltaProb}, with
$$
a_0(X,L)=\delta=\frac{1}{\langle\Bc(P),v_i\rangle+b_i},\; i\in I, \q
a_1(X,L)=-\frac{\wvol(\partial P)}{2\Vol(P)}\max_{i\in I}\langle\wBc(\partial P)-\Bc(P),v_i\rangle \delta^2.
$$
In particular,
    \begin{equation}\label{CompAsymExp}
        \delta_k=\delta-\frac{\wvol(\partial P)}{2\Vol(P)}\max_{i\in I}\langle\wBc(\partial P)-\Bc(P),v_i\rangle \delta^2k^{-1}+O(k^{-2}),
    \end{equation}
where $\wvol$ is the normalized volume of the boundary $\partial P$ and $\wBc$ is the barycenter of $\partial P$ with respect to $\wvol$ $($see \textup{Definition \ref{normalized vol}}$)$.

In fact, for sufficiently large $k$, \eqref{CompAsymExp} is a complete asymptotic expansion of a rational function of $k$ with
both numerator and denominator polynomial of degree $n$ in $k$.
\end{corollary}

In the case of most interest, namely when $X$ is Fano 
polarized by $-K_X$, Theorem \ref{ReflexiveThm}
and \eqref{RTZEq} give a more explicit result.

\begin{corollary}
\label{asym toric fano delta_kMain}
    For a toric Fano (possibly singular) variety $X$ and all $k\in\NN$,
    $$
    \delta_k(-K_X)=\delta(-K_X)-\frac{1}{2k}\delta(-K_X)(1-\delta(-K_X))+O(k^{-2}).
    $$
    Moreover, if $\dim X=2$,
    \begin{equation*}
        \delta_k(-K_X)=\left(1+\frac{(k+1)(2k+1)K_X^2}{4+2k(k+1)K_X^2}\left((\delta(-K_X))^{-1}-1\right)\right)^{-1}.
    \end{equation*}
\end{corollary}

\subsection{Test configurations, higher Donaldson--Futaki invariants,
and Tian's stabilization problem}

Finally, we make contact with the much-studied notion of test configurations \cite{Tian97,Don02}. We show that in the case $P$ is Delzant, the coefficients $a_i(P)$ conjectured by Problem \ref{ExpansionProb}
and established by Theorem \ref{smooth theorem} are closely related to the higher
Donaldson--Futaki invariants of certain test configurations.

\begin{theorem}\label{tc main thm}
Let $P\subset M\otimes_\ZZ\RR$ be a lattice polytope. Let $(X,L)$
be the associated (possibly singular) toric variety given by Theorem \ref{toric-corresp}. Let
$v\in N$ be arbitrary. Then 
$$
\langle a_i(P),v\rangle=DF_i\left(\cX,\cL,\sigma_{q(v)}\right),\q i\in\NN\cup\{0\},
$$
where $\left(\cX,\cL,\sigma_{q(v)}\right)$ denotes a product test configuration
\eqref{ProdTC}
constructed in \S\ref{test config def}. 
\end{theorem}

In particular,
if $(X,L)$ is K-semistable then $DF_1$ vanishes for
all product test configurations \cite[Remark 2.13(2)]{Blu22}, thus $a_1(P)=0$ by
Theorem \ref{tc main thm}, i.e., 
    $\wBc(\partial P)=\Bc(P)$ by Theorem \ref{FirstMainThm},
    recovering a well-known result of Donaldson \cite[Proposition 4.2.1]{Don02}. 
In the reflexive case the additional relation $\wBc(\partial P)=(1+\frac{1}{n})\Bc(P)$ (by Theorems \ref{FirstMainThm} and \ref{Bc_k expansion}) forces $\wBc(\partial P)=\Bc(P)=0$, generalizing one direction of a result of Ono \cite[Corollary 1.7]{Ono11}
to the non-smooth setting (notice Ono assumes $P$ is reflexive
{\it and Delzant}), that is, asymptotic Chow semistability implies $\Bc_k(P)=0$ for all $k$.
By \eqref{RTZEq}, $\Bc_k(P)=0$ if and only if $\delta_k=1$. 
Conversely, if $\delta_k$ stabilizes then by Corollary \ref{asym toric fano delta_kMain} $\delta=1$, and hence
$\delta_k=1$ for large $k$, that by the previous sentence means 
$\Bc_k=0$ for large $k$. In this case we also know by Theorem \ref{FirstMainThm} that $\Bc_k=0$ for all $k$, that in turn by the other direction of Ono's result (only valid if $P$ is also Delzant) implies asymptotic Chow semistability.  In sum we have the following solution of Problem \ref{TianProb} for toric Fanos:

\begin{corollary}
\label{asymChow}
    $\bullet\;$ For a toric Fano manifold $X$, $\delta_k$ stabilizes in $k$ if and only if $(X,-K_X)$ is asymptotically
    Chow semistable. In this case $\delta_k=\delta=1$ for all $k\in\NN$.
    
    $\bullet\;$ For (possibly singular) toric Fano variety $X$, $\delta_k$ stabilizes in $k$ if and only if $\delta_k=\delta=1$ for all $k\in\NN$.
\end{corollary}

Corollary
\ref{asymChow}
is in interesting contrast to the situation for the 
$\alpha_{k,G}$-invariant, that was shown to stabilize always \cite[Theorem 1.4]{JR1}.
Thus $\PP^2$ blown-up at one or two points are examples for which $\delta_k(-K_X)$ does not stabilize, while for the three remaining smooth
 toric del Pezzos stabilization occurs.
This might give the impression that the existence of the \KE metric
dictates stabilization; this is true precisely
when $n\le 6$ (Example \ref{BSExam}).

\begin{remark}
\lb{GeneralRemark}
    The first part of Corollary \ref{asymChow} generalizes to non-toric smooth Fano $X$ using difficult results from algebraic and differential geometry, e.g., the Yau--Tian--Donaldson correspondence, and finite generation. 
    Indeed, when $\delta>1$ this follows from \cite[Corollary 1.11]{Zh21}. When $\delta<1$, here is a sketch: there exists a ``dreamy'' divisor $E$ over $X$ that computes $\delta$ \cite[Theorem 1.2]{LXZ22}, and corresponds to a test configuration $(\cX,\cL)$ with $S_k(E)=\sum_iDF_i(\cX,\cL)k^{-i}$ and $2\Fut(\cX,\cL)=A(E)-S(E)<0$. In particular, $DF_1(\cX,\cL)=-\Fut(\cX,\cL)>0$. Therefore, for sufficiently large $k$, $S_k(E)>S(E)$. Hence $\delta_k<\delta$.
    We emphasize that these results do not seem to give any information concerning the more refined Problem \ref{deltaProb} nor an alternative approach to our
 Corollary \ref{MainAsympThm}.

\end{remark}

\begin{remark}
    Let $(X,L)$ be a polarized toric manifold with associated Delzant lattice polytope $P$. 
    A result of Ono can be interpreted as (he does not implicitly consider the quantized barycenters)  
    $\Bc_k(P)=\Bc(P)$ (for a fixed $k$)
    is a necessary condition for $(X,kL)$ to be Chow semistable \cite[Theorem 1.2]{Ono11} (for Fano, and not necessarily toric manifolds see 
    \cite[Theorem 2.3]{RTZ21} and \cite[Corollary 7.1]{RTZ21}).    
    In particular, if $(X,L)$ is asymptotically Chow semistable, i.e., $(X,kL)$ is Chow semistable for sufficiently large $k$, then Ono's result can be interpreted as implying $\Bc_k(P)=\Bc(P)$ for all $k$ \cite[Theorem 1.4]{Ono11} (by using 
    the existence of the generalized Ehrhart polynomial for homogeneous polynomials (cf. Remark \ref{LoiRemark})). These results are of course special
    cases of Theorems \ref{FirstMainThm} and \ref{tc main thm}, together with the fact that $DF_1\left(\cX,\cL,\sigma_{q(v)}\right)=0$
    for all product test configurations whenever $(X,L)$ is K-semistable \cite[Definition 2.1.2]{Don02}, a condition that
    is implied by asymptotic Chow semistability \cite[\S5]{Thomas05}.
    Resolving a conjecture of Ono, Futaki \cite{Fut12} showed Ono's result can be interpreted in terms of integral invariants related to Hilbert series as in the work of Futaki \cite{FutChow} and Futaki--Ono--Sano \cite{FOS}. See also \cite{Yot16,Yot23,ST19} for more on relations between Chow stability, combinatorics and $\delta_k$-invariants.
    \end{remark}

\begin{problem}
    Does the stabilization of $\Bc_k(P)$ imply asymptotic Chow polystability
    of $X_P$, assuming it is smooth (i.e., $P$ is Delzant by Theorem
\ref{toric-corresp})?
\end{problem}

    The answer is likely affirmative: by a result of Lee--Yotsutani 
    the stabilization of $\Bc_k(P)$ (they phrase this in an equivalent manner) 
    combined with K-semistability for toric degenerations implies asymptotic Chow polystability \cite[Corollary 1.5]{LY24}.
    In general, the stabilization of $\Bc_k(P)$ does not imply asymptotic Chow stability. See \cite[Corollary C.3]{LY24} for a reflexive, non-Delzant example.

\subsection{Organization}

In Section \ref{Bc_k lattice}, we derive a formula for the first order expansion of the quantized barycenters for any lattice polytope (Theorem \ref{FirstMainThm}). In Section \ref{Bc_k smooth}, we express the asympototic expansion of the quantized barycenters for Delzant lattice polytopes in terms of mixed volumes using Hirzebruch--Riemann--Roch (Theorem \ref{smooth theorem}). Section \ref{Bc_k reflexive} simplifies the results in Section \ref{Bc_k lattice} under the additional assumption that the lattice polytope is reflexive (Theorem \ref{Bc_k expansion}). We record in Section \ref{delta_k-invariant for toric del Pezzo surfaces} an explicit computation of the quantized barycenters of reflexive polygons without the use of Ehrhart theory (Theorem \ref{barycenter ratio}). These results are then used to derive the asymptotic behavior of $\delta_k$-invariants in Section \ref{delta_k section} (Corollaries \ref{MainAsympThm} and \ref{asym toric fano delta_kMain}). In Section \ref{tc section} we remark on how the coefficients of the asymptotic expansion are related to higher Donaldson--Futaki invariants (Theorem \ref{tc main thm}). Finally, we provide some examples in Section \ref{example sec}.
In an Appendix it is shown that \eqref{RTZEq} extends to all polarizations \eqref{toric delta_k formula}, generalizing
a formula of Blum--Jonsson and 
Rubinstein--Tian--Zhang.

\bigskip
\noindent
{\bf Acknowledgments.} 
Thanks to 
F. Liu for pointing out the case $n=3$
of Proposition \ref{Ehrhart polynomial reflexive}
and to her and to 
 A. Futaki, K. M\'esz\'aros, Y. Odaka, N. Yotsutani for helpful references, and to
H. Blum, M. Jonsson, 
and K. Zhang for Remark \ref{GeneralRemark}.
Research supported in part 
by NSF grants DMS-1906370,2204347, BSF
grant 2020329, and an
Ann G. Wylie Dissertation Fellowship.

\section{First order expansion of quantized barycenters of lattice polytopes}\label{Bc_k lattice}

In this section we study the first order expansion of $\Bc_k(P)$ for any lattice polytope $P$.

\begin{proposition}
\label{Bc_k asymptotic formula}
    Let $P$ be a lattice polytope. Then
    $$
    \Bc_k(P)=\Bc(P)+\frac{\wvol(\partial P)}{2\Vol(P)}(\wBc(\partial P)-\Bc(P))k^{-1}+O(k^{-2}).
    $$
\end{proposition}
\begin{proof}
    \begin{figure}
        \centering
        \begin{tikzpicture}
            \filldraw[fill=gray!50](0,0)--(4,0)--(5,{sqrt(3)})--(5,{4+sqrt(3)})--(4,4)--cycle;
            \filldraw[fill=gray!25](0,0)--(4,4)--(5,{4+sqrt(3)})--cycle;
            \draw(4,0)--(4,4);
            \draw(4.75,{3*sqrt(3)/4}); 
            \draw(5,{2/3+sqrt(3)});
            \filldraw(0,0)circle(2pt)node[left]{$u_i=-C_i$};
            \filldraw({4/3},0)circle(2pt);
            \filldraw({8/3},0)circle(2pt);
            \filldraw(4,0)circle(2pt)node[right]{$u_i=3-C_i$};
            \filldraw(4.5,{sqrt(3)/2})circle(2pt);
            \filldraw(5,{sqrt(3)})circle(2pt);
            \filldraw({4/3},{4/3})circle(2pt);
            \filldraw({8/3},{4/3})circle(2pt);
            \filldraw(4,{4/3})circle(2pt);
            \filldraw(4.5,{4/3+sqrt(3)/2})circle(2pt);
            \filldraw(5,{4/3+sqrt(3)})circle(2pt);
            \filldraw({8/3},{8/3})circle(2pt);
            \filldraw(4,{8/3})circle(2pt);
            \filldraw(4.5,{8/3+sqrt(3)/2})circle(2pt);
            \filldraw(5,{8/3+sqrt(3)})circle(2pt)node[right]{$\;\;\;\;P_i$};
            \filldraw(4,4)circle(2pt);
            \filldraw(4.5,{4+sqrt(3)/2})circle(2pt);
            \filldraw(5,{4+sqrt(3)})circle(2pt)node[right]{$h=3$};
            \filldraw({8/3+.5},{8/3+sqrt(3)/2})circle(2pt);
            \draw[->](3,-.5)--(3,-1.5)node[midway,right]{$\pi$};
            \filldraw[fill=gray!25](0,-3)--(4,-3)--(5,{sqrt(3)-3})--cycle;
            \filldraw(0,-3)circle(2pt)node[left]{$u_i=-C_i$};
            \filldraw({4/3},-3)circle(2pt);
            \filldraw({8/3},-3)circle(2pt);
            \filldraw(4,-3)circle(2pt)node[right]{$u_i=3-C_i$};
            \filldraw(4.5,{sqrt(3)/2-3})circle(2pt);
            \filldraw(5,{sqrt(3)-3})circle(2pt);
            \filldraw({8/3+.5},{-3+sqrt(3)/2})circle(2pt);
            \draw[dotted]({4/3},0)--({4/3},{4/3});
            \draw[dotted]({8/3},0)--({8/3},{8/3});
            \draw[dotted](4.5,{sqrt(3)/2})--(4.5,{sqrt(3)/2+4});
        \end{tikzpicture}
        \caption{The lattice polytope $P_i$: $u_i+C_i$ lattice points above each lattice
        point in $P$ with $i$-th coordinate equal to $u_i$. The dotted lines are fibers of $\pi:M\times\RR\ni (u,h)\mapsto u\in M$.}\label{lattice polytope}
    \end{figure}
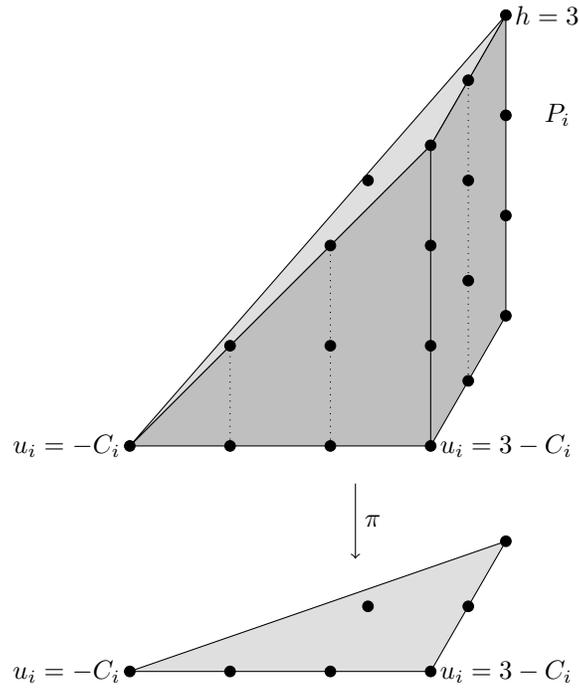
    Fix $1\leq i\leq n$. Pick an integer $C_i\geq-\min\{u_i:(u_1,\ldots,u_n)\in P\}$. Then for any $u\in P$, the function $h_i(u):=u_i+C_i$ is non-negative. Consider the lattice polytope
    (see Figure \ref{lattice polytope})
    \begin{equation}\label{no translation}
        P_i:=\{(u,h)\in\RR^n\times\RR~:~u\in P,~0\leq h\leq h_i(u)\}.
    \end{equation}
    Notice that
    \begin{align*}
        E_{P_i}(k)&=\sum_{u\in P\cap\frac{1}{k}L}\#\left\{\left(u,0\right),\left(u,\frac{1}{k}\right),\cdots,\left(u,h_i(u)\right)\right\}\\
        &=\sum_{u\in P\cap\frac{1}{k}L}(k(u_i+C_i)+1)\\
        &=(kC_i+1)E_P(k)+k\sum_{u\in P\cap\frac{1}{k}L}u_i.
    \end{align*}
    Therefore,
    \begin{equation}
    \label{qbc coordinate}
        \frac{1}{E_P(k)}\sum_{u\in P\cap\frac{1}{k}L}u_i=\frac{E_{P_i}(k)}{kE_P(k)}-C_i-\frac{1}{k}.
    \end{equation}
    Now,
    \begin{equation}\label{volP_i}
        \Vol\left(P_i\right)=\int_Ph_i(u)d\Vol=\int_Pu_id\Vol+C_i\Vol(P).
    \end{equation}
    Notice that the projection to $P\subseteq\RR^n\times\{0\}$ from the upper boundary of $P_i$, denoted by
    \begin{equation*}
        H_i:=\{(u,h)\in P_i~:~h=h_i(u)\},
    \end{equation*}
    preserves normalized volumes, since this projection restricted to the lattice points on the hyperplanes is an isomorphism of lattices. In particular, $\Vol(P)=\wvol(H_i)$.
    It follows that
    \begin{equation}\label{nvolP_i}
        \wvol\left(\partial P_i\right)=\int_{\partial P}h_i(u)d\wvol+2\Vol(P)=\int_{\partial P}u_id\wvol+C_i\wvol(\partial P)+2\Vol(P).
    \end{equation}
    Combining \eqref{qbc coordinate}, \eqref{volP_i}, \eqref{nvolP_i}, and Theorem \ref{Ehrhart polynomial},
    \begin{align*}
        \frac{1}{E_P(k)}\sum_{u\in P\cap\frac{1}{k}L}u_i&=\frac{\Vol\left(P_i\right)k^{n+1}+\frac{1}{2}\wvol\left(\partial P_i\right)k^n+O(k^{n-1})}{\Vol(P)k^{n+1}+\frac{1}{2}\wvol(\partial P)k^n+O(k^{n-1})}-C_i-\frac{1}{k}\\
        &=\frac{\Vol\left(P_i\right)}{\Vol(P)}-C_i+\left(\frac{\Vol(P)\wvol\left(\partial P_i\right)-\wvol(\partial P)\Vol\left(P_i\right)}{2\Vol(P)^2}-1\right)k^{-1}+O(k^{-2})\\
        &=\frac{\int_Pu_id\Vol}{\Vol(P)}+\frac{\wvol(\partial P)}{2\Vol(P)}\left(\frac{\int_{\partial P}u_id\wvol}{\wvol(\partial P)}-\frac{\int_Pu_id\Vol}{\Vol(P)}\right)k^{-1}+O(k^{-2}),
    \end{align*}
concluding the proof of Proposition \ref{Bc_k asymptotic formula}.
    \end{proof}

\begin{remark}
\lb{LoiRemark}
An alternative proof of Proposition \ref{Bc_k asymptotic formula} can be obtained from the first-order
expansion of the {\it sum} of the function 
$\phi(u)=u$ over the integral lattice points
(as opposed to $\Bc_k$ that is its {\it average}), see,
e.g., \cite[(5)]{Yot23}, \cite[(7)]{LoiZuddas}.
Such a formula follows directly from the
existence of the generalized Ehrhart polynomial  
for homogeneous polynomials and 
from Ehrhart--MacDonald reciprocity for such polynomials
that generalizes Proposition \ref{reciprocity}
\cite[Proposition 4.1]{BrionVergne}. 
While such a proof would be perhaps more standard for experts in Ehrhart theory, the proof we give is certainly more elementary and allows us to make the connection to the {\it rooftop polytope}
(that, in the toric geometry literature, goes back to Donaldson \cite[\S4.2]{Don02}) that then leads to the connection to Donaldson--Futaki invariants in Theorem 
\ref{tc main thm}.
\end{remark}

\begin{corollary}
\label{Bc_k=0 for all k if large k}
    Each coordinate of the quantized barycenter $\Bc_k(P)$ is a rational function of $k$. More precisely, it is the quotient of two polynomials of degree at most $n$. In particular, if $\Bc_k(P)$ is constant for $n+1$ values of $k$, then $\Bc_k(P)=\Bc(P)$ for all $k$.  
    Consequently, if $\Bc_k(P)$ is constant for sufficiently large $k$, then $\Bc_k(P)=\Bc(P)$ for all k.
\end{corollary}
\begin{proof}
    Let $\Bc_{k,i}$ denote the $i$-th coordinate of the quantized barycenter $\Bc_k(P)\in\RR^n$.
    Using the same notation as in the proof of Proposition \ref{Bc_k asymptotic formula}, by (\ref{qbc coordinate}), we have
    \begin{align}
        \Bc_{k,i}
        =\frac{1}{E_P(k)}\sum_{u\in P\cap\frac{1}{k}L}u_i
        =\frac{E_{P_i}(k)-(C_ik+1)E_P(k)}{kE_P(k)}.\label{Bc_k no translation}
    \end{align}
    Notice that the numerator
    \begin{align*}
        Q_i(k):=E_{P_i}(k)-(C_ik+1)E_P(k)=E_{P_i}(k)-E_{P\times[0,C_i]}(k),
    \end{align*}
    is a polynomial with respect to $k$ of degree at most $n+1$ and with constant term $0$ by Theorem \ref{Ehrhart polynomial}. Hence $Q_i(k)=k\widetilde{Q}_i(k)$ for some polynomial $\widetilde{Q}_i$ of degree at most $n$.
    Now,
    $$
        \Bc_{k,i}=\frac{\widetilde{Q}_i\left(k\right)}{E_P\left(k\right)}.
    $$
    Since there are $n+1$ values of $k$ such that $\Bc_{k,i}$ is equal to a constant, say, $\lambda_i$, the polynomial
    $$
        \widetilde{Q}_i\left(k\right)-\lambda_iE_P\left(k\right)
    $$
    has degree at most $n$ and has at least $n+1$ zeros. It follows that this polynomial vanishes everywhere, i.e., $\Bc_{k,i}$ is constant for all $k$.
\end{proof}

\begin{proof}[Proof of Theorem \ref{FirstMainThm}]
    The second half of the theorem is proved by Proposition \ref{Bc_k asymptotic formula} and Corollary \ref{Bc_k=0 for all k if large k}. It remains to show \eqref{translation}. Notice that for any $i$,
    $$
        P\left(c;i\right)-c=\left\{\left(u,h\right)\,:\,u\in P+c,\; h\in\left[0,u_i\right]\right\}-c=\left\{\left(u,h\right)\,:\,u\in P,\; h\in\left[0,u_i+c_i\right]\right\}=P_i
    $$
    (Recall \eqref{no translation}). In particular, their Ehrhart polynomials coincide. This reduces \eqref{translation} to \eqref{Bc_k no translation}.
\end{proof}

\section{Asymptotics of quantized barycenters of Delzant lattice polytopes}\label{Bc_k smooth}

In this section, we express the asympototic expansion of the quantized barycenters in terms of mixed volumes of virtual polytopes using toric geometry. Here we restrict ourselves to Delzant lattice polytopes, namely the lattice polytopes that determines a smooth fan, i.e., the primitive generators of any $n$-dimensional cone form a $\ZZ$-basis. Recall Definition \ref{virtual def}.

\begin{lemma}\label{mixed volumes}
    For divisors $D_1,\ldots,D_n$ with associated virtual polytopes $P_1,\ldots,P_n$,
    $$
        V\left(P_1,\ldots,P_n\right)=\frac{1}{n!}D_1\cdots D_n.
    $$
\end{lemma}
\begin{proof}
    Since any divisor can be written as the difference of two ample divisors, by multi-linearity we may assume $D_1,\ldots,D_n$ are ample. In that case the corresponding virtual polytopes $P_1,\ldots,P_n$ are polytopes. For $k_1,\ldots,k_n>0$, the divisor $k_1D_1+\cdots+k_nD_n$ is ample with associated polytope $k_1P_1+\cdots+k_nP_n$. Therefore
    $$
        \Vol\left(k_1P_1+\cdots+k_nP_n\right)=\frac{1}{n!}\left(k_1D_1+\cdots+k_nD_n\right)^n.
    $$
    Both sides are polynomials in $k_1,\ldots,k_n$. In particular, the coefficients of $k_1\cdots k_n$ in both sides of the equation should be equal, i.e.,
    $
        n!V\left(P_1,\ldots,P_n\right)=D_1\cdots D_n.
    $
\end{proof}

\begin{definition}
    Define the Todd class
    $$
        \Td\left(X\right)=\prod_{i=1}^n\frac{\xi_i}{1-e^{-\xi_i}},
    $$
    where the $\xi_i$'s are the Chern roots of the tangent bundle $T_X$, i.e., they are symbolic objects such that
    $$
        c\left(T_X\right)=\prod_{i=1}^n\left(1+\xi_i\right),
    $$
    or more explicitly,
    $$
        c_i\left(T_X\right)=\sigma_i\left(\xi_1,\ldots,\xi_n\right),
    $$
    where $\sigma_i$ is the $i$-th elementary symmetric polynomial.
\end{definition}

$\Td(X)$ can be expressed in terms of the cohomology classes of the torus invariant divisors.

\begin{definition}
    Define
    $$
        \Td\left(x\right):=\frac{x}{1-e^{-x}}=\sum_{j=0}^\infty\frac{B_j}{j!}x^j=1+\frac{1}{2}x+\frac{1}{12}x^2-\frac{1}{720}x^4+\cdots,
    $$
    where $B_j$ is called the $j$-th Bernoulli number.
\end{definition}
\begin{lemma}
{\rm\cite[Theorem 13.1.6]{CLS11}}
    Let $D_i$ denote the divisor corresponding to $v_i$. Then
    $$
        \Td\left(X\right)=\prod_{i=1}^d\Td\left(\left[D_i\right]\right).
    $$
\end{lemma}

\begin{proposition}
    Suppose $L$ is ample. Then
    $$
        h^0\left(X,kL\right)=\sum_{j=0}^na_jk^j,
    $$
    where
    $$
        a_j=\frac{L^j}{j!}\cdot\sum_{\sum\limits_{i=1}^d\ell_i=n-j}\prod_{i=1}^d\frac{B_{\ell_i}}{\ell_i!}D_i^{\ell_i}.
    $$
    In particular,
    \begin{align*}
        a_n&=\frac{L^n}{n!},&a_{n-1}&=\frac{L^{n-1}\cdot\left(-K_X\right)}{2\left(n-1\right)!},&a_{n-2}&=\frac{L^{n-2}\cdot\left(\left(-K_X\right)^2+\sum\limits_{i<j}D_i\cdot D_j\right)}{12\left(n-2\right)!}.
    \end{align*}
\end{proposition}
\begin{proof}
    By the Kodaira vanishing theorem, $h^0(X,kL)=\chi(X,kL)$ for sufficiently large $k$. Since both are polynomials, they are equal for all $k$. Therefore
    \begin{align*}
        h^0\left(X,kL\right)&=\chi\left(X,kL\right)\\
        &=\int_X\ch\left(kL\right)\cdot\Td\left(X\right)\\
        &=\int_Xe^{k\left[L\right]}\cdot\prod_{i=1}^d\Td\left(\left[D_i\right]\right)\\
        &=\int_X\sum_{j=0}^n\frac{k^j}{j!}\left[L\right]^j\cdot\prod_{i=1}^d\sum_{\ell=0}^\infty\frac{B_\ell}{\ell!}\left[D_i\right]^\ell\\
        &=k^j\sum_{j=0}^n\frac{L^j}{j!}\cdot\sum_{\sum\limits_{i=1}^d\ell_i=n-j}\prod_{i=1}^d\frac{B_{\ell_i}}{\ell_i!}D_i^{\ell_i}.
    \end{align*}
\end{proof}

\begin{remark}\label{scalar curvature}
    Recall the average scalar curvature
    $$
        \overline{S}:=\frac{\int_XS_\omega\cdot\omega^n}{\int_X\omega^n}=\frac{\int_Xn\Ric\omega\wedge\omega^{n-1}}{\int_X\omega^n}=\frac{n\left(-K_X\right)\cdot L^{n-1}}{L^n}=\frac{2a_{n-1}}{a_n}.
    $$
    By Theorem \ref{Ehrhart polynomial},
    $$
        \overline{S}=\frac{\wvol\left(\partial P\right)}{\Vol\left(P\right)},
    $$
    as observed by Donaldson \cite[p. 309]{Don02}.
\end{remark}

Next, using the construction of Proposition \ref{Bc_k asymptotic formula}, we relate $\Bc_k(P)$ to this expansion.

\begin{corollary}\label{combRR cor}
    Let $P\subseteq M\otimes\RR$ be a Delzant lattice polytope. Fix $v\in N$. Let $(X,L)$ denote the corresponding polarized toric manifold as in Theorem \ref{toric-corresp}, where
    $$
        L=\sum_{i=1}^db_iD_i.
    $$
    Also recall \eqref{cD def}. Then
    $$
        \left\langle\Bc_k\left(P\right),v\right\rangle=\frac{\sum\limits_{j=0}^nc'_{j+1}k^j}{\sum\limits_{j=0}^na_jk^j}=\sum_{j=0}^\infty k^{-j}\sum_{i=0}^jc'_{n+1-j+i}\left(\frac{\delta_{i0}}{a_n}+\sum_{\ell=1}^i\frac{\left(-1\right)^\ell}{a_n^{\ell+1}}\sum_{\sum\limits_{m=1}^\ell i_m=i-\ell}\prod_{m=1}^\ell a_{n-i_m-1}\right),
    $$
    where
    \begin{align*}
        a_j&=\frac{L^j}{j!}\cdot\sum_{\sum\limits_{i=1}^d\ell_i=n-j}\prod_{i=1}^d\frac{B_{\ell_i}}{\ell_i!}D_i^{\ell_i},&c'_j&=\frac{\overline{\cL}'^j}{j!}\cdot\sum_{\sum\limits_{i=1}^d\ell_i=n+1-j}\prod_{i=1}^d\frac{B_{\ell_i}}{\ell_i!}\cD_i^{\ell_i},
    \end{align*}
    and
    $$
        \overline{\cL}'=\sum_{i=1}^db_i\cD_i.
    $$
    In particular,
    \begin{align*}
        a_n&=\frac{L^n}{n!},\\
        a_{n-1}&=\frac{L^{n-1}\cdot\left(-K_X\right)}{2\left(n-1\right)!},\\
        a_{n-2}&=\frac{L^{n-2}\cdot\left(\left(-K_X\right)^2+\sum\limits_{i<j}D_i\cdot D_j\right)}{12\left(n-2\right)!},\\
        c'_{n+1}&=\frac{{\overline{\cL}'}^{n+1}}{\left(n+1\right)!},\\
        c'_n&=\frac{{\overline{\cL}'}^n\cdot\left(-K_{\overline{\cX}/\PP^1}\right)}{2n!},\\
        c'_{n-1}&=\frac{{\overline{\cL}'}^{n-1}\cdot\left(\left(-K_{\overline{\cX}/\PP^1}\right)^2+\sum\limits_{1\leq i<j\leq d}\cD_i\cdot\cD_j\right)}{12\left(n-1\right)!},
    \end{align*}
    where
    $$
        -K_{\overline{\cX}/\PP^1}:=-K_{\overline{\cX}}+\pi^*K_{\PP^1}=-K_{\overline{\cX}}-\cD_{d+1}-\cD_{d+2}=\sum_{i=1}^d\cD_i,
    $$
    and
    \begin{multline*}
        \left\langle\Bc_k\left(P\right),v\right\rangle=\frac{c'_{n+1}}{a_n}+\left(\frac{c'_n}{a_n}-\frac{c'_{n+1}a_{n-1}}{a_n^2}\right)\frac{1}{k}\\
        +\left(\frac{c'_{n-1}}{a_n}-\frac{c'_na_{n-1}+c'_{n+1}a_{n-2}}{a_n^2}+\frac{c'_{n+1}a_{n-1}^2}{a_n^3}\right)\frac{1}{k^2}+O\left(\frac{1}{k^3}\right).
    \end{multline*}
\end{corollary}
\begin{proof}
We use the setup in Section \ref{Delzant subsec}. By \eqref{qbc coordinate},
$$
    \left\langle\Bc_k\left(P\right),v\right\rangle=\frac{E_{P_{v,q}}\left(k\right)}{kE_P\left(k\right)}-q-\frac{1}{k}.
$$
Write
\begin{align*}
    E_P\left(k\right)&=\sum_{j=0}^na_jk^j,\\
    E_{P_{v,q}}\left(k\right)&=\sum_{j=0}^{n+1}c_jk^j.
\end{align*}
We have
\begin{align*}
    a_j&=\frac{L^j}{j!}\cdot\sum_{\sum\limits_{i=1}^d\ell_i=n-j}\prod_{i=1}^d\frac{B_{\ell_i}}{\ell_i!}D_i^{\ell_i},\\
    c_j&=\frac{\overline{\cL}^j}{j!}\cdot\sum_{\sum\limits_{i=1}^{d+2}\ell_i=n+1-j}\prod_{i=1}^{d+2}\frac{B_{\ell_i}}{\ell_i!}\cD_i^{\ell_i}=qa_{j-1}+a_j+\frac{\overline{\cL}'^j}{j!}\cdot\sum_{\sum\limits_{i=1}^d\ell_i=n+1-j}\prod_{i=1}^d\frac{B_{\ell_i}}{\ell_i!}\cD_i^{\ell_i}.
\end{align*}
For simplicity, we define
$$
    a_j:=0
$$
for $j<0$ or $j>n$, and
$$
    c_j:=0
$$
for $j<0$ or $j>n+1$. We also define
$$
    c'_j:=c_j-qa_{j-1}-a_j=\frac{\overline{\cL}'^j}{j!}\cdot\sum_{\sum\limits_{i=1}^d\ell_i=n+1-j}\prod_{i=1}^d\frac{B_{\ell_i}}{\ell_i!}\cD_i^{\ell_i}.
$$
Note $c_j$ depends on the choice of $q$ while $c_j'$ is canonical.
Then
\begin{align*}
    \left\langle\Bc_k\left(P\right),v\right\rangle&=\frac{E_{P_{v,q}}\left(k\right)}{kE_P\left(k\right)}-q-\frac{1}{k}\\
    &=\frac{E_{P_{v,q}}\left(k\right)-qkE_P\left(k\right)-E_P\left(k\right)}{kE_P\left(k\right)}\\
    &=\frac{\sum\limits_{j=0}^{n+1}\left(c_j-qa_{j-1}-a_j\right)k^j}{\sum\limits_{j=0}^na_jk^{j+1}}\\
    &=\frac{\sum\limits_{j=0}^{n+1}c'_jk^j}{\sum\limits_{j=0}^na_jk^{j+1}}\\
    &=\sum_{j=0}^\infty k^{-j}\sum_{i=0}^jc'_{n+1-j+i}\left(\frac{\delta_{i0}}{a_n}+\sum_{\ell=1}^i\frac{\left(-1\right)^\ell}{a_n^{\ell+1}}\sum_{\sum\limits_{m=1}^\ell i_m=i-\ell}\prod_{m=1}^\ell a_{n-i_m-1}\right).
\end{align*}

Finally, recall that $a_0=c_0=1$. Therefore $c'_0=c_0-a_0=0$, and
$$
    \left\langle\Bc_k\left(P\right),v\right\rangle=\frac{\sum\limits_{j=1}^{n+1}c'_jk^j}{\sum\limits_{j=0}^na_jk^{j+1}}=\frac{\sum\limits_{j=0}^nc'_{j+1}k^j}{\sum\limits_{j=0}^na_jk^j}.
$$
This concludes the proof.
\end{proof}

Theorem \ref{smooth theorem} is a reformulation of Corollary \ref{combRR cor} in terms of mixed volumes of virtual polytopes.
\begin{proof}[Proof of Theorem \ref{smooth theorem}]
    This follows from Corollary \ref{combRR cor} and Lemma \ref{mixed volumes}.
\end{proof}

\section{First order expansion of quantized barycenters of reflexive polytopes}\label{Bc_k reflexive}

In this section, we study the special case that the lattice polytope is reflexive.

\begin{definition}{\rm \cite[Definition 2.3.12]{CLS11}}
\label{reflexive polytope}
    A lattice polytope in $M\otimes\RR$ is called {\it reflexive} if it contains the origin, and each of its facet is given by $\langle u,v\rangle=1$ where $v$ is some primitive vector of $N$.
\end{definition}

\begin{proposition}{\rm \cite{Mac71}, \cite[Theorem 19.1]{Gruber}}\label{reciprocity}
    For $k\in\NN$,
    $$
        E_P\left(-k\right)=\left(-1\right)^n\#\left(\left(\Int kP\right)\cap M\right).
    $$
\end{proposition}

\begin{corollary}\label{reflexive reciprocity}
    Let $P\subset M_\RR$ be an $n$-dimensional reflexive polytope. Then
    $$
        E_P\left(-k\right)=\left(-1\right)^nE_P\left(k-1\right).
    $$
\end{corollary}
\begin{proof}
    By Definition \ref{refdef}, there exist $v_1,\ldots,v_d\in N$ so that $P$ can be written
    \begin{equation}\label{Pref}
        P=\bigcap_{i=1}^d\left\{u\in M_\RR\,:\,\langle u,v_i\rangle\geq-1\right\}.
    \end{equation}
    Then for $k\in\NN$,
    \begin{align*}
        (\Int kP)\cap M&=\bigcap_i\left\{u\in M\,:\,\langle u,v_i\rangle>-k\right\}\\
        &=\bigcap_i\left\{u\in M\,:\,\langle u,v_i\rangle\geq-k+1\right\}\\
        &=(k-1)P\cap M.
    \end{align*}
    By Theorem \ref{reciprocity},
    $$
        E_P(-k)=(-1)^n\#((\Int kP)\cap M)=(-1)^n\#((k-1)P\cap M)=(-1)^nE_P(k-1),
    $$
    as claimed.
\end{proof}

\begin{proposition}
\label{Ehrhart polynomial reflexive}
    Let $P$ be an $n$-dimensional reflexive polytope. Then $\wvol(\partial P)=n\Vol(P)$. In other words,
    $$
    E_P(k)=\Vol(P)k^n+\frac{n}{2}\Vol(P)k^{n-1}+\cdots+1.
    $$
    In particular,
    $$
    E_P(k)=\begin{cases}
        \Vol(P)k^2+\Vol(P)k+1,&n=2;\\
        \Vol(P)k^3+\frac{3}{2}\Vol(P)k^2+\left(\frac{1}{2}\Vol(P)+2\right)k+1,&n=3.
    \end{cases}
    $$
\end{proposition}
This proposition is well-known \cite{Ehr67a,Gruber}, except the case
$n=3$ that seems new.

\begin{proof}
    By Corollary \ref{reflexive reciprocity}, since
    $$
        E_P(-k)=(-1)^n\Vol(P)k^n+(-1)^{n-1}\wvol(\partial P)k^{n-1}+\cdots+1,
    $$
    and
    $$
        (-1)^nE_P(k-1)=(-1)^n\Vol(P)(k-1)^n+(-1)^n\wvol(\partial P)(k-1)^{n-1}+\cdots+(-1)^n,
    $$
    comparing the coefficients we get
    $$
        (-1)^{n-1}\wvol(\partial P)=(-1)^{n-1}n\Vol(P)+(-1)^n\wvol(\partial P),
    $$
    i.e.,
    $$
        \wvol(\partial P)=\frac{n}{2}\Vol(P).
    $$
    It follows that when $n=2$,
    $$
        E_P(k)=\Vol(P)k^2+\Vol(P)k+1.
    $$
    When $n=3$,
    \begin{equation}\label{reciprocity n=3}
        E_P(k)=\Vol(P)k^3+\frac{3}{2}\Vol(P)k^2+a_1k+1
    \end{equation}
    for some coefficient $a_1$. To determine $a_1$, we may simply plug in special values for $k$. By Corollary \ref{reflexive reciprocity},
    $$
        E_P(-1)=-E_P(0)=-1.
    $$
    Letting $k=-1$ in \eqref{reciprocity n=3} gives
    $$
        -1=\frac{1}{2}\Vol(P)-a_1+1,
    $$
    i.e.,
    $$
        a_1=\frac{1}{2}\Vol(P)+2,
    $$
    as claimed.
\end{proof}

\begin{theorem}\label{Bc_k expansion}
    Let $P$ be reflexive. Then
    $$
        \Bc_k(P)=\Bc(P)\left(1+\frac{1}{2k}\right)+O(k^{-2}).
    $$
\end{theorem}
\begin{proof}
    Recall \eqref{Pref}. Fix $1\leq i\leq d$. Consider the reflexive polytope
    $$
        P_i:=\{(u,h)\in\RR^n\times\RR~:~u\in P,\ -1\leq h\leq\langle u,v_i\rangle+1\}.
    $$
    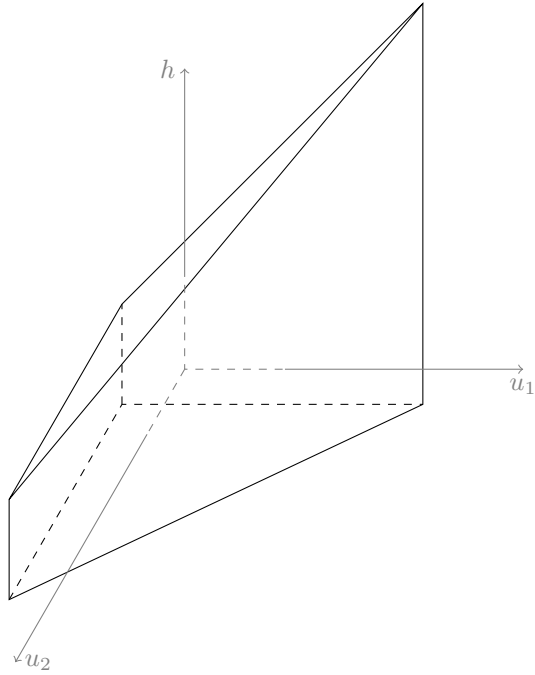
\begin{figure}
        \centering
        \begin{tikzpicture}
            \draw(-4/3-1,{-sqrt(3)})--(-4/3-1,{-4/3-sqrt(3)})--(8/3+1/2,{-4/3+1/2*sqrt(3)})--(8/3+1/2,{4+1/2*sqrt(3)})--cycle--(-4/3+1/2,{1/2*sqrt(3)})--(8/3+1/2,{4+1/2*sqrt(3)});
            \draw[dashed](-4/3-1,{-4/3-sqrt(3)})--(-4/3+1/2,{-4/3+1/2*sqrt(3)})--(8/3+1/2,{-4/3+1/2*sqrt(3)});
            \draw[dashed](-4/3+1/2,{-4/3+1/2*sqrt(3)})--(-4/3+1/2,{1/2*sqrt(3)});
            \draw[gray,dashed](0,4/3)--(0,0)--(4/3,0);
            \draw[gray,->](4/3,0)--(4.5,0)node[below]{$u_1$};
            \draw[gray,->](0,4/3)--(0,4)node[left]{$h$};
            \draw[gray,dashed](0,0)--(-1/2,{-1/2*sqrt(3)});
            \draw[gray,->](-1/2,{-1/2*sqrt(3)})--(-2.25,{-2.25*sqrt(3)})node[right]{$u_2$};
        \end{tikzpicture}
        \caption{The polytope $P_1$ corresponding to $\PP^2$ with $v_1=(1,0)$, $v_2=(0,1)$, $v_3=(-1,-1)$.}
    \end{figure}
    Notice that
    \begin{align*}
        E_{P_i}(k)&=\sum_{u\in P\cap\frac{1}{k}L}\#\left\{\left(u,-1\right),\left(u,-1+\frac{1}{k}\right),\cdots,\left(u,\langle u,v_i\rangle+1\right)\right\}\\
        &=\sum_{u\in P\cap\frac{1}{k}L}(k\langle u,v_i\rangle+2k+1)\\
        &=(2k+1)E_P(k)+k\sum_{u\in P\cap\frac{1}{k}L}\langle u,v_i\rangle\\
        &=(2k+1)E_P(k)+kE_P(k)\langle \Bc_k(P),v_i\rangle.
    \end{align*}
    Therefore,
    \begin{equation}
    \label{qbc coordinate inner}
        \langle \Bc_k(P),v_i\rangle=\frac{E_{P_i}(k)}{kE_P(k)}-2-\frac{1}{k}.
    \end{equation}
    Now,
    \begin{equation}\label{volP_i-reflexive}
        \Vol\left(P_i\right)=\int_P(\langle u,v_i\rangle+2)d\Vol=(\langle \Bc(P),v_i\rangle+2)\Vol(P).
    \end{equation}
    By Theorem \ref{Ehrhart polynomial reflexive},
    \begin{align*}
        \langle \Bc_k(P),v_i\rangle&=\frac{\Vol\left(P_i\right)k^{n+1}+\frac{n+1}{2}\Vol\left(P_i\right)k^n+O(k^{n-1})}{\Vol(P)k^{n+1}+\frac{n}{2}\Vol(P)k^n+O(k^{n-1})}-2-\frac{1}{k}\\
        &=\frac{\Vol\left(P_i\right)}{\Vol(P)}\left(1+\frac{1}{2k}\right)-2-\frac{1}{k}+O(k^{-2})\\
        &=\langle\Bc(P),v_i\rangle\left(1+\frac{1}{2k}\right)+O(k^{-2}).
    \end{align*}
Since $P$ is bounded, $\{v_i\}_{i=1}^d$ spans $N\otimes\RR$. This completes the proof.
\end{proof}

\section{Quantized barycenters of reflexive polygons}
\label{delta_k-invariant for toric del Pezzo surfaces}

In this section, we give an explicit formula for the quantized barycenters of reflexive polygons without using Ehrhart theory.
 
\begin{theorem}\label{barycenter ratio}
    Let $P$ be a reflexive polygon, i.e., $P\subset\RR^2$ is a lattice polytope given by
    \begin{equation}\label{polygon defn}
        P=\bigcap_{j=1}^d\{x\in M_\mathbb{R}~:~\langle x,v_j\rangle\geq-1\}
    \end{equation}
    for some primitive vectors $v_j\in N$. Then
    \begin{equation}
    \label{Bc_k in dim=2}
        \Bc_k(P)
        =\frac{(k+1)(2k+1)\wvol(\partial P)}{4+2k(k+1)\wvol(\partial P)}\Bc(P).
    \end{equation}
\end{theorem}
\begin{proof}
    Since $\dim P=2$, each facet of $P$ is a line segment on the line $\langle x,v_j\rangle=-1$. Let 
    $$F_j^{(1)}$$ denote this facet. We may index the lattice points $F_j^{(1)}\cap M$ on the polygon $\partial P$ counterclockwise as $u_{j,0}^{(1)},\ldots,u_{j,N_j}^{(1)}$. Note that $u_{j,0}^{(1)}$ and $u_{j,N_j}^{(1)}$ are vertices of $P$ and will be indexed twice as they belong to two facets. Define $w_j:=u_{j,1}^{(1)}-u_{j,0}^{(1)}$, i.e., $w_j$ characterizes the distance between two neighboring lattice points on a line perpendicular to $v_j$. For any $0\leq l\leq N_j$,
    \begin{equation}\label{vertex dist}
    u_{j,l}^{(1)}=u_{j,0}^{(1)}+lw_j.
    \end{equation}

    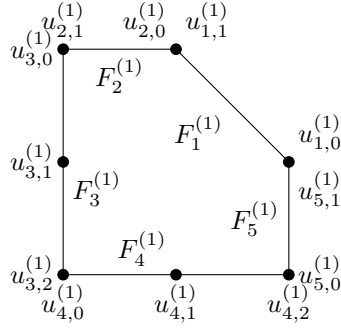
\begin{figure}
        \centering
        \begin{tikzpicture}
            \draw(1.5,0)node[above right]{$u_{1,0}^{(1)}$}--(0,1.5)node[above right]{$u_{1,1}^{(1)}$}node[midway,below left]{$F_1^{(1)}$}node[above left]{$u_{2,0}^{(1)}$}--(-1.5,1.5)node[above]{$u_{2,1}^{(1)}$}node[midway,below]{$F_2^{(1)}$}node[left]{$u_{3,0}^{(1)}$}--(-1.5,-1.5)node[midway,left]{$u_{3,1}^{(1)}$}node[left]{$u_{3,2}^{(1)}$}node[midway,below right]{$F_3^{(1)}$}node[below]{$u_{4,0}^{(1)}$}--(1.5,-1.5)node[midway,below]{$u_{4,1}^{(1)}$}node[below]{$u_{4,2}^{(1)}$}node[midway,above left]{$F_4^{(1)}$}node[right]{$u_{5,0}^{(1)}$}--cycle node[below right]{$u_{5,1}^{(1)}$}node[midway,left]{$F_5^{(1)}$};
            \filldraw(-1.5,-1.5)circle(2pt);
            \filldraw(-1.5,0)circle(2pt);
            \filldraw(-1.5,1.5)circle(2pt);
            \filldraw(0,-1.5)circle(2pt);
            \filldraw(0,1.5)circle(2pt);
            \filldraw(1.5,-1.5)circle(2pt);
            \filldraw(1.5,0)circle(2pt);
        \end{tikzpicture}
        \caption{The indexing of the edges and the lattice points of the polygon corresponding to the blow-up of $\PP^1\times\PP^1$.}
    \end{figure}
    
    We repeat this construction for any $k\in\NN$. Each facet of $kP$ is a line segment on the line $\langle x,v_j\rangle=-k$ and are denoted by $$F_j^{(k)}.$$ Since $kP$ is a dilation of $P$,
    $$
    F_j^{(k)}=k\cdot F_j^{(1)},
    $$
    and in particular, the end points are $k\cdot u_{j,0}^{(1)}$ and $k\cdot u_
    {j,N_j}^{(1)}=k\cdot u_{j,0}^{(1)}+kN_j\cdot w_j$ (recall \eqref{vertex dist}). We may also index the lattice points $F_j^{(k)}\cap M$ counterclockwise as $u_{j,0}^{(k)},\ldots,u_{j,kN_j}^{(k)}$, where for any $0\leq l\leq kN_j$,
    \begin{equation}\label{lattice pt dist}
    u_{j,l}^{(k)}=u_{j,0}^{(k)}+lw_j=k\cdot u_{j,0}^{(1)}+lw_j.
    \end{equation}

    \begin{figure}
        \centering
        \begin{tikzpicture}
            \filldraw(-3,0)circle(2pt);
            \filldraw(-2,0)circle(2pt);
            \filldraw(-1,0)circle(2pt);
            \filldraw(0,0)circle(2pt);
            \filldraw(1,0)circle(2pt);
            \filldraw(2,0)circle(2pt);
            \filldraw(3,0)circle(2pt);
            \filldraw(-3,1)circle(2pt);
            \filldraw(-2,1)circle(2pt);
            \filldraw(-1,1)circle(2pt);
            \filldraw(0,1)circle(2pt);
            \filldraw(1,1)circle(2pt);
            \filldraw(2,1)circle(2pt);
            \filldraw(3,1)circle(2pt);
            \filldraw(-3,2)circle(2pt);
            \filldraw(-2,2)circle(2pt);
            \filldraw(-1,2)circle(2pt);
            \filldraw(0,2)circle(2pt);
            \filldraw(1,2)circle(2pt);
            \filldraw(2,2)circle(2pt);
            \filldraw(3,2)circle(2pt);
            \filldraw(-3,3)circle(2pt);
            \filldraw(-2,3)circle(2pt);
            \filldraw(-1,3)circle(2pt);
            \filldraw(0,3)circle(2pt);
            \filldraw(1,3)circle(2pt);
            \filldraw(2,3)circle(2pt);
            \filldraw(3,3)circle(2pt);
            \filldraw(5,0)circle(2pt);
            \filldraw(-5,5)circle(2pt);
            \draw[->](-3.5,0)--(3.5,0);
            \draw[->](0,-.5)--(0,3.5);
            \draw(1,0)node[below]{$u_{j,0}^{(1)}$}--(-1,1)node[left]{$u_{j,1}^{(1)}$};
            \draw(2,0)node[below]{$u_{j,0}^{(2)}$}--(-2,2)node[left]{$u_{j,2}^{(2)}$};
            \draw(3,0)node[below]{$u_{j,0}^{(3)}$}--(-3,3)node[left]{$u_{j,3}^{(3)}$};
            \draw(5,0)node[below]{$u_{j,0}^{(k)}$}--(-5,5)node[left]{$u_{j,k}^{(k)}$};
            \draw(0,0)node[below left]{$0$};
            \draw(4,0)node{$\cdots$};
            \draw(-4,4)node{$\ddots$};
        \end{tikzpicture}
        \caption{The relationship between $F_j^{(1)}$ and a general $F_j^{(k)}$.}
        \label{fig:enter-label}
    \end{figure}
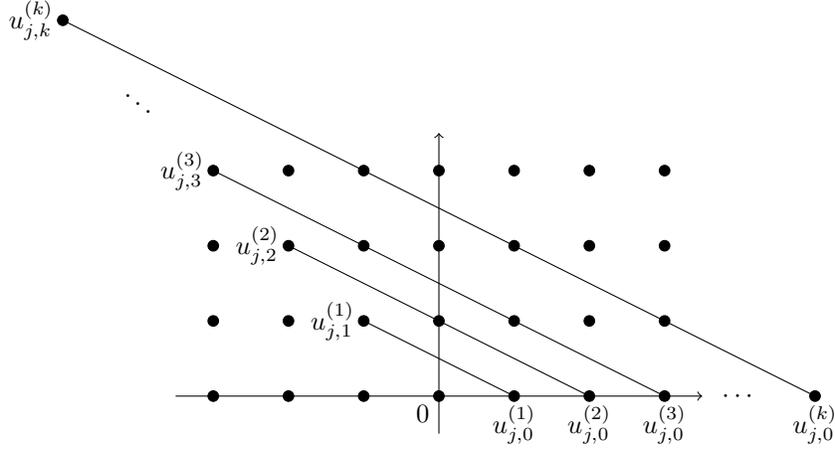
    
    Notice that the lattice points of $kP$ are
    \begin{align}
    kP\cap M&=\{0\}\cup\bigcup_{i=1}^{k}(\partial(iP)\cap M)\label{split bdry}\\
    &=\{0\}\cup\bigcup_{i=1}^{k}\bigcup_{j=1}^{d}\{u_{j,0}^{(i)},\ldots,u_{j,iN_j}^{(i)}\}.\label{split bdry points}
    \end{align}
    Indeed, $0\in kP\cap M$. For any $1\leq i\leq k$, $\partial(iP)\subseteq kP$. Hence
    $$
    kP\cap M\supseteq\{0\}\cup\bigcup_{i=1}^{k}(\partial(iP)\cap M).
    $$
    On the other hand, suppose there exists $x\in kP\cap M$ other than the origin that does not belong to any $\partial(iP)$. Then there is a unique integer $0\leq k_x<k$ such that $x\notin k_xP$ and $x\in\Int(k_x+1)P$. Recall \eqref{polygon defn}. Since $x\notin k_xP$, there is some $j$ such that $\langle x,v_j\rangle<-k_x$. On the other hand, since $x\in\Int(k_x+1)P$, we know $\langle x,v_j\rangle>-k_x-1$. However, $\langle x,v_j\rangle\in\ZZ$, which is a contradiction.

    We are now ready to compute the barycenters using \eqref{split bdry points}. Recall that in the expression \eqref{split bdry points} every vertex of $iP$ shows up twice.

    Let $m_j^{(k)}$ denote the midpoint of $F_j^{(k)}$, i.e.,
    \begin{equation}\label{m defn}
    m_j^{(k)}:=u_{j,0}^{(k)}+\frac{kN_j}{2}w_j.
    \end{equation}
    Since $F_j^{(k)}=k\cdot F_j^{(1)}$, we also have
    \begin{equation}\label{m relatn}
    m_j^{(k)}=k\cdot m_j^{(1)}.
    \end{equation}
    Recall the indexed lattice points
    $$
    F_j^{(k)}\cap M=\{u_{j,0}^{(k)},\ldots,u_{j,kN_j}^{(k)}\}.
    $$
    Since each vertex is indexed twice,
    \begin{equation}\label{number of pts}
        \#(\partial(kP)\cap M)=\sum_{j=1}^d\left(\#(F_j^{(k)}\cap M)-1\right)=k\sum_{j=1}^dN_j.
    \end{equation}
    Also, by \eqref{lattice pt dist}, \eqref{m defn}, and \eqref{m relatn},
    \begin{align*}
    \sum_{u\in\partial(kP)\cap M}u&=\sum_{j=1}^d\left(\frac{1}{2}u_{j,0}^{(k)}+\sum_{l=1}^{kN_j-1}u_{j,l}^{(k)}+\frac{1}{2}u_{j,kN_j}^{(k)}\right)\\
    &=\sum_{j=1}^d\left(\frac{1}{2}u_{j,0}^{(k)}+\sum_{l=1}^{kN_j-1}\left(u_{j,0}^{(k)}+l\cdot w_j\right)+\frac{1}{2}\left(u_{j,0}^{(k)}+kN_j\cdot w_j\right)\right)\\
    &=k\sum_{j=1}^dN_j\left(u_{j,0}^{(k)}+\frac{kN_j}{2}w_j\right)\\
    &=k\sum_{j=1}^dN_jm_j^{(k)}\\
    &=k^2\sum_{j=1}^dN_jm_j^{(1)}.
    \end{align*}
    Combining these two equations and \eqref{split bdry}, we get
    \begin{align*}
    \Bc_k(P)&=\frac{1}{k}\Bc(kP\cap M)\\
    &=\frac{1}{k\#(kP\cap M)}\sum_{u\in kP\cap M}u\\
    &=\frac{1}{k\left(1+\sum\limits_{i=1}^k\#(\partial(iP)\cap M)\right)}\sum_{i=1}^k\sum_{u\in\partial(kP)\cap M}u\\
    &=\frac{(k+1)(2k+1)}{6\left(1+\frac{k(k+1)}{2}\sum\limits_{j=1}^dN_j\right)}\sum_{j=1}^dN_jm_j^{(1)},
    \end{align*}
    and
    $$
    \Bc(P)=\lim_{k\to\infty}\Bc_k(P)=\frac{2\sum\limits_{j=1}^dN_jm_j^{(1)}}{3\sum\limits_{j=1}^dN_j}.
    $$
    Therefore
    $$
    \Bc_k(P)=\frac{(k+1)(2k+1)\sum\limits_{j=1}^dN_j}{4+2k(k+1)\sum\limits_{j=1}^dN_j}\Bc(P).
    $$
    Finally, by Definition \ref{normalized vol},  $N_j=\wvol(F_j^{(1)})$ (The line segment $F_j^{(1)}$ has $N_j$ fundamental parallelotopes). Then
    \begin{equation}\label{polygon wvol}
        \sum_{j=1}^dN_j=\wvol\left(\partial P\right).
    \end{equation}
    This concludes the proof.
\end{proof}

\begin{remark}
    We record here also the proof of Theorem \ref{barycenter ratio} using Ehrhart theory. Recall the proof of Theorem \ref{asym toric fano delta_k}. By Theorem \ref{Ehrhart polynomial reflexive},
    \begin{align*}
        E_{P_i}(k)&=\Vol(P_i)k^3+\frac{3}{2}\Vol(P_i)k^2+\left(\frac{1}{2}\Vol(P_i)+2\right)k+1,\\
        E_P(k)&=\Vol(P)k^2+\Vol(P)k+1.
    \end{align*}
    By \eqref{qbc coordinate inner} and \eqref{volP_i-reflexive},
    \begin{align*}
        \langle\Bc_k(P),v_i\rangle&=\frac{E_{P_i}(k)-(2k+1)E_P(k)}{kE_P(k)}\\
        &=\frac{\left(\Vol(P_i)-2\Vol(P)\right)k^3+\left(\frac{3}{2}\Vol(P_i)-3\Vol(P)\right)k^2+\left(\frac{1}{2}\Vol(P_i)-\Vol(P)\right)k}{k(\Vol(P)k^2+\Vol(P)k+1)}\\
        &=\frac{k^2+\frac{3}{2}k+\frac{1}{2}}{\Vol(P)k^2+\Vol(P)k+1}\left(\Vol(P_i)-2\Vol(P)\right)\\
        &=\frac{(k+1)(2k+1)\Vol(P)}{2+2k(k+1)\Vol(P)}\langle\Bc(P),v_i\rangle.
    \end{align*}
    The result follows since by Proposition \ref{Ehrhart polynomial reflexive},$\Vol(P)=\frac{1}{2}\wvol(\partial P)$.
\end{remark}

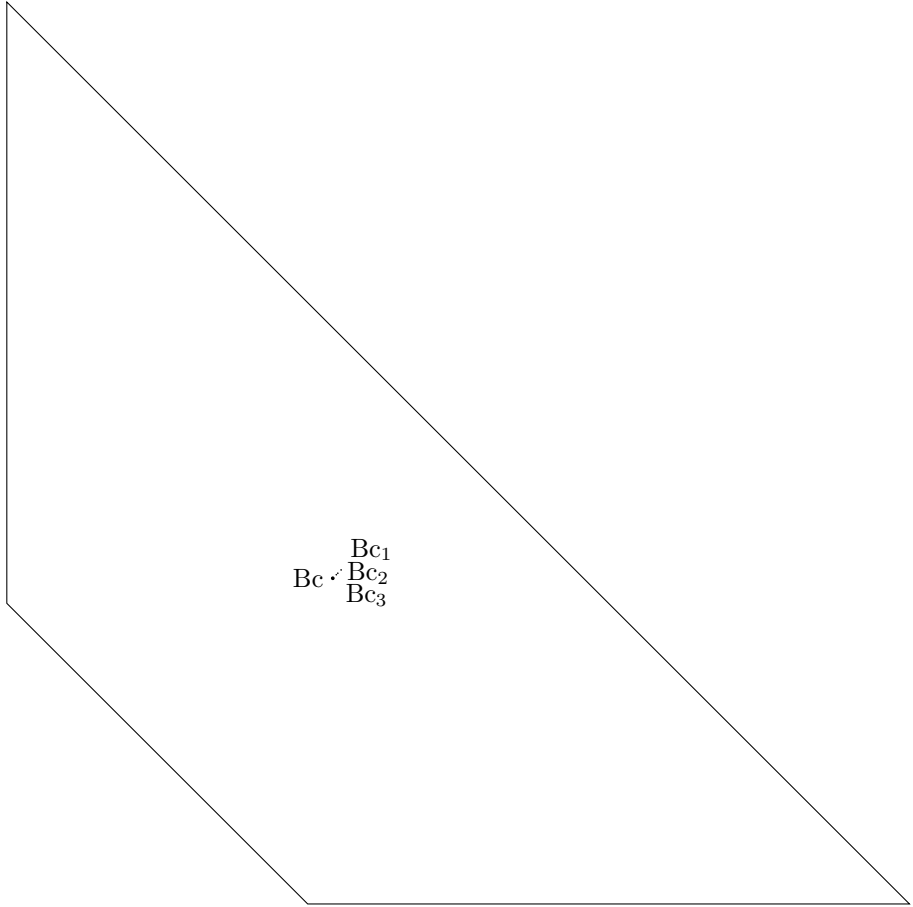
\begin{figure}
    \centering
    \begin{tikzpicture}
        \draw(8,-4)--(0,-4)--(-4,0)--(-4,8)--cycle;
        \filldraw({4/9},{4/9})circle(.1pt)node[above right]{$\Bc_1$};
        \filldraw(.4,.4)circle(.1pt)node[right]{$\Bc_2$};
        \filldraw({8/21},{8/21})circle(.1pt)node[below right]{$\Bc_3$};
        \filldraw({1/3},{1/3})circle(.5pt)node[left]{$\Bc$};
    \end{tikzpicture}
    \caption{The barycenter and quantized barycenters of the polygon corresponding to $\FF_1$. In particular, $\Bc_1=(\frac{1}{9},\frac{1}{9})$, $\Bc_2=(\frac{1}{10},\frac{1}{10})$, $\Bc_3=(\frac{2}{21},\frac{2}{21})$, and $\Bc=(\frac{1}{12},\frac{1}{12})$.}
\end{figure}
\begin{figure}
    \centering
    \begin{tikzpicture}
        \draw(4,0)--(0,4)--(-4,4)--(-4,-4)--(4,-4)--cycle;
        \filldraw(-.5,-.5)circle(.1pt)node[below left]{$\Bc_1$};
        \filldraw({-5/11},{-5/11})circle(.1pt)node[left]{$\Bc_2$};
        \filldraw({-56/129},{-56/129})circle(.1pt)node[above left]{$\Bc_3$};
        \filldraw({-8/21},{-8/21})circle(.5pt)node[right]{$\Bc$};
    \end{tikzpicture}
    \caption{The barycenter and quantized barycenters of the polygon corresponding to the blow-up of $\PP_1\times\PP_1$. In particular, $\Bc_1=(-\frac{1}{8},-\frac{1}{8})$, $\Bc_2=(-\frac{5}{44},-\frac{5}{44})$, $\Bc_3=(-\frac{14}{129},-\frac{14}{129})$, and $\Bc=(-\frac{2}{21},-\frac{2}{21})$.}
\end{figure}
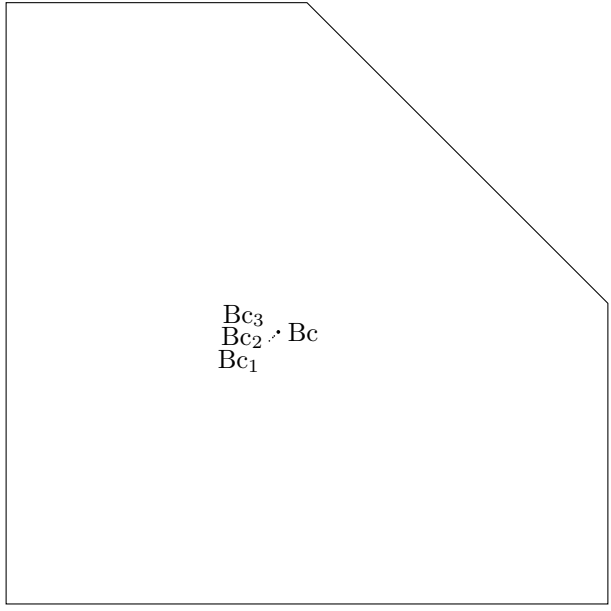

\begin{remark}
\label{rmk equivalence of Bc=0}
    As a consequence of Theorem \ref{barycenter ratio}, we see that the following are equivalent for toric del Pezzo surfaces:
    \begin{enumerate}[$\rm 1)$]
        \item $\Bc(P)=0$;
        \item $\Bc_k(P)=0$ for some $k$;
        \item $\Bc_k(P)=0$ for all $k$.
    \end{enumerate}
This strengthens Corollary
\ref{Bc_k=0 for all k if large k} that holds in all dimensions.
Moreover, this resolves Problem \ref{stabProb}
in dimension $n=2$ showing the answer is `no'.
    In addition, formula (\ref{Bc_k in dim=2}) tells us that the quantized barycenter $\Bc_k(P)$ lies on a certain line.
\end{remark}

\begin{remark}
\label{rmk higher dim toric}
    For higher dimensions, it is harder to derive such a formula this way, and in general $\Bc_k$ and $\Bc$ are not collinear. See Example \ref{not colinear}. In fact, for a polytope of dimension $3$ or higher, it has faces of dimension $2$ or higher, whose shapes and lattice points are more difficult to characterize.
\end{remark}

\section{Asymptotics of \texorpdfstring{$\delta_k$}{delta k}-invariants}\label{delta_k section}

In this section we study the asymptotics of $\delta_k$-invariant using the results of the previous sections.

\begin{theorem}\label{rmk delta_k=1 if k large}
    If $\delta_k(-K_X)=1$ for sufficiently large $k$, then $\delta_k(-K_X)=1$ for all $k$.    
\end{theorem}
\begin{proof}
    By (\ref{toric delta_k formula}), $\delta_k(-K_X)=1$ if and only if $\Bc_k(P)=0$, where $P$ is the polytope associated to $X$. By Corollary \ref{Bc_k=0 for all k if large k}, $\Bc_k(P)=0$ for all $k$.
\end{proof}

\begin{proof}[Proof of Corollary \ref{MainAsympThm}]
    By Proposition \ref{Bc_k asymptotic formula}, (\ref{toric delta}) and (\ref{toric delta_k formula}),
    \begin{align*}
        \delta_k^{-1}
        &=\max_{i=1,\ldots,d}(\langle \Bc_k(P),v_i\rangle+b_i)     \\
        &=\max_{i=1,\ldots,d}\left(\langle \Bc(P),v_i\rangle+b_i +\frac{\wvol(\partial P)}{2\Vol(P)}\langle(\wBc(\partial P)-\Bc(P)),v_i\rangle k^{-1} +O(k^{-2})\right) \\
        &=\delta^{-1}+\frac{\wvol(\partial P)}{2\Vol(P)}\max_{i\in I}\langle\wBc(\partial P)-\Bc(P),v_i\rangle k^{-1} +O(k^{-2}).
    \end{align*}
    Therefore,
    \begin{align*}
        \delta_k-\delta
        &=\frac{\delta^2(\delta^{-1}-\delta_k^{-1})}{\delta\delta_k^{-1}}\\
        &=\dfrac{\delta^2(\delta^{-1}-\delta_k^{-1})}{1-\delta(\delta^{-1}-\delta_k^{-1})}    \\
        &=\delta^2\left(\frac{1}{\delta}-\frac{1}{\delta_k}\right)+O(k^{-2}) \\
        &=-\delta^2\frac{\wvol(\partial P)}{2\Vol(P)}\max_{i\in I}\langle\wBc(\partial P)-\Bc(P),v_i\rangle k^{-1}+O(k^{-2}).
    \end{align*}

    Since $\delta_k$ is the minimum of finitely many rational functions, for sufficiently large $k$, there is a rational function that dominates the rest. Indeed, consider the asymptotic expansion and take the one whose coefficient dominates in the lexicographic order. Therefore $\delta_k$ is a rational function of $k$ when $k$ is sufficiently large.
\end{proof}

\begin{remark}
    Recall
    $$
        \delta_k=\min_{i=1,\ldots,d}\frac{1}{\left\langle\Bc_k(P),v_i\right\rangle+b_i}.
    $$
    Since $\Bc_k(P)$ has an asymptotic expansion by Corollary \ref{Bc_k=0 for all k if large k}, for each $i=1,\ldots,d$,
    $$
        \frac{1}{\left\langle\Bc_k(P),v_i\right\rangle+b_i}
    $$
    has an asymptotic expansion. However, $\argmin_{i=1,\ldots,d}\left(\left\langle\Bc_k(P),v_i\right\rangle+b_i\right)^{-1}$ is not necessarily constant for all $k$. Therefore we can only deduce that there is an expansion that works for sufficiently large $k$, but not necessarily all $k$.
\end{remark}

\begin{theorem}
\label{asym toric fano delta_k}
    Let $X$ be Fano. Then
    $$
    \delta_k(-K_X)=\delta(-K_X)-\frac{1}{2k}\delta(-K_X)(1-\delta(-K_X))+O(k^{-2}).
    $$
\end{theorem}
\begin{proof}
    By Theorem \ref{Bc_k expansion},
    \begin{align*}
        \delta_k^{-1}&=\max_{i=1,\ldots,d}\langle \Bc_k(P),v_i\rangle+1\\
        &=\max_{i=1,\ldots,d}\langle\Bc(P),v_i\rangle\left(1+\frac{1}{2k}\right)+1+O(k^{-2})\\
        &=\delta^{-1}+
        \frac{1}{2k}\left(\delta^{-1}-1\right)+O(k^{-2}).
    \end{align*}
    Therefore
    \begin{align*}
        \delta_k&=\frac{\delta}{1+\frac{1}{2k}(1-\delta)+O(k^{-2})}\\
        &=\delta-\frac{1}{2k}\delta(1-\delta)+O(k^{-2}).
    \end{align*}
\end{proof}

\begin{lemma}
    For a smooth del Pezzo surface $X$,
    $$
        h^0(X,-K_X)=1+K_X^2.
    $$
    In particular, if $X$ is toric with associated polytope $P$, then
    $$
        \wvol\left(\partial P\right)=K_X^2.
    $$
\end{lemma}
\begin{proof}
    Since $-K_X$ is ample, by Kodaira Vanishing Theorem
    $\chi(\cO_X(-K_X))=h^0(X,-K_X)$ and by Riemann--Roch \cite[p. 472]{GH14}
    $
    \chi(\cO_X(-K_X))=\chi(\cO_X)+K_X^2=h^0(X,\cO_X)-h^1(X,\cO_X)+h^2(X,\cO_X)+K_X^2= 1-h^{0,1}_{\bar\partial}+h^{0,2}_{\bar\partial}+K_X^2=1+K_X^2
    $.
    
    By \eqref{split bdry}, \eqref{number of pts} and \eqref{polygon wvol},
    $
        h^0(X,-K_X)=\#(P\cap M)=1+\sum_{j=1}^dN_j=1+\wvol\left(\partial P\right).
    $
\end{proof}

\begin{theorem}
\label{thm delta_k for toric del Pezzo}
Let $X$ be a smooth toric del Pezzo surface. Then
\begin{equation}
        \delta_k(-K_X)=\left(1+\frac{(k+1)(2k+1)K_X^2}{4+2k(k+1)K_X^2}\left((\delta(-K_X))^{-1}-1\right)\right)^{-1},
\end{equation}
In particular, 
\begin{equation*}
    \frac{\delta_k(-K_X)}{\delta(-K_X)}=1-\frac{1}{2k}(1-\delta(-K_X))+O\left(k^{-2}\right).
\end{equation*}
\end{theorem}
\begin{proof}
    By Theorem \ref{barycenter ratio},
    \begin{align*}
        \delta_k(-K_X)^{-1}-1
        &=\max_{i=1,\ldots,d}\langle \Bc_k(P),v_i\rangle\\
        &=\frac{(k+1)(2k+1)K_X^2}{4+2k(k+1)K_X^2}\max_{i=1,\ldots,d}\langle \Bc(P),v_i\rangle\\
        &=\frac{(k+1)(2k+1)K_X^2}{4+2k(k+1)K_X^2}\left(\delta(-K_X)^{-1}-1\right).
    \end{align*}
    Therefore
    \begin{align*}
        \delta_k(-K_X)&=\left(1+\frac{(k+1)(2k+1)K_X^2}{4+2k(k+1)K_X^2}\left((\delta(-K_X))^{-1}-1\right)\right)^{-1}\\
        &=\left(1+\frac{1+\frac{3}{2k}+O(k^{-2})}{1+\frac{1}{k}+O(k^{-2})}\left((\delta(-K_X))^{-1}-1\right)\right)^{-1}\\
        &=\left((\delta(-K_X))^{-1}+\frac{1}{2k}\left((\delta(-K_X))^{-1}-1\right)O(k^{-2})\right)^{-1}\\
        &=\left(1-\frac{1}{2k}(1-\delta(-K_X))+O\left(k^{-2}\right)\right)\delta(-K_X),
    \end{align*}
    as claimed.
\end{proof}

\section{Relation to test configurations}
\label{tc section}

In this section, we explain how the coefficients of the asymptotic expansion relate to the higher Donaldson--Futaki invariants.
\begin{definition}
    \begin{itemize}
        \item A test configuration for $(X,L)$ is a triple $(\cX,\cL,\sigma)$, where $\sigma$ is a $\CC^*$-action on the pair $(\cX,\cL)$, with a $\CC^*$-equivariant projective morphism $\pi:\cX\to\CC$ such that $\cL$ is $\pi$-relatively very ample, i.e., $\pi$ lifts to $\cX\hookrightarrow\PP^r_{\CC^1}=\Proj\CC[t][Z_0,\ldots,Z_r]$ which pulls $\cO(1)$ back to $\cL$, and each fiber over $t\in\CC^*\subseteq\CC$, denoted by $(\cX_t,\cL_t)$, is isomorphic to $(X,L)$.
        \item We say that a test configuration is a product configuration if $\cX\cong X\times\CC$, and a trivial configuration if the $\CC^*$ action on $\cX\cong X\times\CC$ is trivial on the first factor.
        \item The $\CC^*$-action on $\cL$ induces a $\CC^*$-action on $H^0(\cX_0,\cL_0^k)$. Let $w_k$ denote the sum of the weights of this action.
        \item Define the Donaldson--Futaki coefficients to be the coefficients of the expansion
        $$
            \frac{w_k}{kH^0\left(\cX_0,\cL_0^k\right)}=\sum_{i=0}^\infty DF_i\left(\cX,\cL,\sigma\right)k^{-i}.
        $$
    \end{itemize}
\end{definition}

\subsection{Construction of the rooftop test configuration}\label{test config def}
    We construct a product test configuration
    \beq\lb{ProdTC}
    (\cX,\cL,\sigma_{q(v)})
    \eeq
    as follows.
    Recall the pair $(\overline{\cX},\overline{\cL})$ constructed in \eqref{cXcL}. The projection $M_\RR\times\RR\to\RR$ corresponds to an equivariant map $\pi:\overline{\cX}\to\PP^1$ (see Figure \ref{test configuration}). The subfan consisting of cones generated by $(v_1,0),\ldots,(v_d,0),(0,1)$ (i.e., the cones on or above the hyperplane $M_\RR\times\{0\}$) is the product of the fan of $X$ and the fan of $\CC$, and hence corresponds to the toric manifold $\pi^{-1}(\CC)\cong X\times\CC$. Similarly, the subfan consisting of cones generated by $(v_1,0),\ldots,(v_d,0),(v,-1)$ (i.e., the cones on or below the hyperplane $M_\RR\times\{0\}$) corresponds to $\pi^{-1}(\PP^1\setminus\{0\})\cong X\times\CC$. Let $\cX$ denote the latter open submanifold and $\cL$ the restriction of $\overline{\cL}$ over $\cX$. Note that this product structure implies that each fiber of $\pi$ is a copy of $(X,L)$.
    
    The lattice point $(0,-1)\in N\times\ZZ=\Hom(\CC^*,(\CC^*)^{n+1})$ induces a $\CC^*$-action expressed on the torus $(\CC^*)^{n+1}=(\CC^*)^{n}\times \CC^*$ (the first factor is the open-orbit
    in $X$ considered as a toric variety) as
    $$
        \sigma(t):z\mapsto\sigma(t).z:=\left(z_1,\ldots,z_n,t^{-1}z_{n+1}\right),
        \q t\in\CC^*.
    $$
     By definition of a toric variety (both $X$ and $X\times \CC$ are such) it is enough to specify the action on the open orbit and it will extend uniquely to $X\times \CC$. Of course, it is not obvious what the action does at $t=0$ as $1/t$ blows up; this shows in passing that what we construct is not a trivial, but rather a product, test configuration (this test configuration is trivial if and only if $\overline{\cX}=X\times\PP^1$, which is precisely when $v=0$).
     We can lift it to a $\CC^*$-action on the line bundle $\overline{\cL}$ by (recall $q=q(v)$ defined in \eqref{qEq})
    $$
        \sigma_{q(v)}(t):\left(z,f\right)\mapsto\left(\sigma(t).z,t^{-q}f\right).
    $$
    More precisely, for a section $s_{u'}\in H^0(\overline{\cX},\overline{\cL})$ given by $z\mapsto z^{u'}$ under the standard local trivialization over $(\CC^*)^{n+1}$, where $u'\in M\times\ZZ$,
    \begin{align*}
        \sigma_{q(v)}(t)\left(z,z^{u'}\right)&=\left(\left(z_1,\ldots,z_n,t^{-1}z_{n+1}\right),t^{-q}z_1^{u'_1}\cdots z_n^{u'_n}z_{n+1}^{u'_{n+1}}\right)\\
        &=\left(\left(z_1,\ldots,z_n,t^{-1}z_{n+1}\right),t^{u'_{n+1}-q}z_1^{u'_1}\cdots z_n^{u'_n}\left(t^{-1}z_{n+1}\right)^{u'_{n+1}}\right)\\
        &=\left(\sigma_t\left(z\right),t^{u'_{n+1}-q}\sigma_t\left(z\right)^{u'}\right).
    \end{align*}
    Thus,
    $$
        \sigma_{q(v)}(t)\left(s_{u'}\right)=t^{u'_{n+1}-q}s_{u'}\circ\sigma_t,
    $$
    i.e., a section $s_{u'}$ has weight $u'_{n+1}-q$.

    We define this triple $(\cX,\cL,\sigma_{q(v)})$ to be the product test configuration \eqref{ProdTC}. This test configuration is also studied in \cite[Remark 2.4]{Sai}, \cite[Proposition 2.2.2]{Don02}.

\subsection{Quantized barycenters and higher Donaldson--Futaki
invariants}

\begin{theorem}
    Let $\left(\cX,\cL,\sigma_v\right)$ be the product test configuration \eqref{ProdTC} constructed in \S\ref{test config def}. Then
    $$
        \left\langle\Bc_k\left(P\right),v\right\rangle=\sum_{i=0}^\infty DF_i\left(\cX,\cL,\sigma_{q(v)}\right)k^{-i}.
    $$
\end{theorem}
\begin{proof}
    Consider the lattice points on the facet
    $$
        \left\{\left(u,\left\langle u,v\right\rangle+q\right)\in M_\RR\times\RR\,:\,u\in P\cap k^{-1}M\right\}.
    $$
    They correspond to the torus invariant sections of $\overline{\cL}^k$ that do not vanish along $\cX_0=\cD_{d+2}$ \cite[p. 61]{Ful93}.
    For $u\in P\cap k^{-1}M$, let $s_u$ denote the section of $\overline{\cL}^k$ corresponding to $(u,\langle u,v\rangle+q)$, which has weight $k\langle u,v\rangle$. It follows that the section $s_u|_{\cL_0^k}$ on $\cL_0^k$ also has weight $k\langle u,v\rangle$. Since $\{s_u|_{\cL_0^k}\}_{u\in P\cap k^{-1}M}$ spans $H^0(\cX_0,\cL_0^k)$, taking the sum of all their weights, we get
    $$
        \frac{w_k}{kH^0\left(\cX_0,\cL_0^k\right)}=\frac{1}{kE_P\left(k\right)}\sum_{u\in P\cap k^{-1}M}k\left\langle u,v\right\rangle=\left\langle\Bc_k\left(P\right),v\right\rangle.
    $$
    This concludes the proof.
\end{proof}

Combining this with Corollary \ref{combRR cor}, we obtain an expression for the higher Donaldson--Futaki invariants of this test configuration.
\begin{corollary}
    Assume $X$ is smooth. Let $\left(\cX,\cL,\sigma_v\right)$ be the product test configuration \eqref{ProdTC} constructed in \S\ref{test config def}. Then
    $$
        DF_j\left(\cX,\cL,\sigma_{q(v)}\right)=\sum_{i=0}^jc'_{n+1-j+i}\Bigg(\frac{\delta_{i0}}{a_n}+\sum_{\ell=1}^i\frac{\left(-1\right)^\ell}{a_n^{\ell+1}}\sum_{\sum\limits_{m=1}^\ell i_m=i-\ell}\prod_{m=1}^\ell a_{n-i_m-1}\Bigg),
    $$
    where $a_j$ and $c'_j$ are defined in Corollary \ref{combRR cor}.
\end{corollary}

\section{Examples}\label{example sec}

\subsection{Two simple polygon examples}

\begin{example}\label{example1 of normalized vol}
    Let $P$ be the associated polytope of $\PP^2$. Specifically, the three edges of $\partial P$ are $F_1:=\{(-1,y)\in\RR^2~:~-1\leq y\leq2\}$,
    $F_2:=\{(x,-1)\in\RR^2~:~-1\leq x\leq2\}$, and 
    $F_3:=\{(x,y)\in\RR^2~:~x+y-1=0,~-1\leq x\leq2\}$.
    Let $C_1$, $C_2$, and $C_3$ denote the standard cubes of the induced affine sublattices on $F_1$, $F_2$, and $F_3$, respectively.
    Precisely, we may choose $C_1:=\{(-1,y)\in\RR^2~:~-1\leq y\leq0\}$,
    $C_2:=\{(x,-1)\in\RR^2~:~-1\leq x\leq0\}$ and 
    $C_3:=\{(x,y)\in\RR^2~:~x+y-1=0,~-1\leq x\leq0\}$. 
    Then $\Vol(C_1)=\Vol(C_2)=1$ and $\Vol(C_3)=\sqrt{2}$. 
    Thus one has
    \begin{align*}
        \wvol\left(\partial P\right)
        &=\wvol\left(F_1\right)+\wvol\left(F_2\right)+\wvol\left(F_3\right)          \\
        &=\frac{\Vol\left(F_1\right)}{\Vol\left(C_1\right)} +\frac{\Vol\left(F_2\right)}{\Vol\left(C_2\right)}+\frac{\Vol\left(F_3\right)}{\Vol\left(C_3\right)}  \\
        &=3+3+\frac{3\sqrt{2}}{\sqrt{2}}\\
        &=9,\\
        \wBc\left(\partial P\right)
        &=\frac{1}{\wvol\left(\partial P\right)}\left(\int_{F_1}\left(x,y\right)d\wvol+\int_{F_2}\left(x,y\right)d\wvol+\int_{F_3}\left(x,y\right)d\wvol\right)\\
        &=\frac{1}{9}\left(\left(-3,\frac{3}{2}\right)+\left(\frac{3}{2},-3\right)+\left(\frac{3}{2},\frac{3}{2}\right)\right)\\
        &=0.
    \end{align*}
    By Remark \ref{scalar curvature}, the average scalar curvature
    $
        \overline{S}=\frac{\wvol\left(\partial P\right)}{\Vol\left(P\right)}=2.
    $
\end{example}
\begin{example}\label{example2 of normalized vol}
    Let $P'$ be the associated polytope of blow-up of $\PP^2$ at one point. Specifically, the four edges of $\partial P'$ by $F_1':=\{(-1,y)\in\RR^2~:~0\leq y\leq2\}$,
    $F_2':=\{(x,-1)\in\RR^2~:~0\leq x\leq2\}$, 
    $F_3':=\{(x,y)\in\RR^2~:~x+y=1,~-1\leq x\leq2\}$, and 
    $F_4':=\{(x,y)\in\RR^2~:~x+y=-1,~-1\leq x\leq0\}$.
    Let $C_1'$, $C_2'$, $C_3'$, and $C_4'$ denote the standard cube of the induced lattice on $F_1'$, $F_2'$, $F_3'$, and $F_4'$, respectively.
    Then $\Vol(C_1')=\Vol(C_2')=1$ and $\Vol(C_3)=\Vol(C_4')=\sqrt{2}$.
    Thus one has
    \begin{align*}
        \wvol\left(\partial P'\right)
        &=\wvol\left(F_1'\right)+\wvol\left(F_2'\right)+\wvol\left(F_3'\right)+\wvol\left(F_4'\right)          \\
        &=\frac{\Vol\left(F_1'\right)}{\Vol\left(C_1'\right)} +\frac{\Vol\left(F_2'\right)}{\Vol\left(C_2'\right)}+\frac{\Vol\left(F_3'\right)}{\Vol\left(C_3'\right)}+\frac{\Vol\left(F_4'\right)}{\Vol\left(C_4'\right)}  \\
        &=2+2+\frac{3\sqrt{2}}{\sqrt{2}}+\frac{\sqrt{2}}{\sqrt{2}}\\
        &=8,\\
        \wBc\left(\partial P'\right)
        &=\frac{1}{\wvol\left(\partial P'\right)}\left(\int_{F_1'}\left(x,y\right)d\wvol+\int_{F_2'}\left(x,y\right)d\wvol+\int_{F_3'}\left(x,y\right)d\wvol+\int_{F_4'}\left(x,y\right)d\wvol\right)\\
        &=\frac{1}{8}\left(\left(-2,2\right)+\left(2,-2\right)+\left(\frac{3}{2},\frac{3}{2}\right)+\left(-\frac{1}{2},-\frac{1}{2}\right)\right)\\
        &=\left(\frac{1}{8},\frac{1}{8}\right).
    \end{align*}
    By Remark \ref{scalar curvature}, the average scalar curvature
    $
        \overline{S}=\frac{\wvol\left(\partial P\right)}{\Vol\left(P\right)}=2.
$
\end{example}

\subsection{A 3-dimensional reflexive example
with non-collinear quantized barycenters}

\begin{example}\label{not colinear}
    The Fano 3-fold $No.3.29$ has a toric structure with the primitive elements of the rays of its fan being the columns of the matrix  \cite[p. 43]{WW82}
    \begin{align*}
        \begin{pmatrix}
        1 &0 &0 &1  &0  &-1  \\
        0 &1 &0 &-1 &-1 &-1  \\
        0 &0 &1 &0  &0  &-1 
    \end{pmatrix}.
    \end{align*}
    Then a computer calculation shows
    \begin{align*}
        \Bc_1(P)&=(\frac{6}{28},\frac{-12}{28},\frac{3}{28}), \\
        \Bc_2(P)&=(\frac{51}{260},\frac{-99}{260},\frac{24}{260})     \\
        \Bc_3(P)&=(\frac{201}{1071},\frac{-387}{1071},\frac{93}{1071}).
    \end{align*}
    It follows that $\Bc_k$ are not on a line. This means 
    that $a_{i}(P)\not=0$ for some $i>1$.
\end{example}

\subsection{Virtual rooftop polytopes and verification of the terms
in the expansion}

\begin{example}
    Consider $X=\PP^1$ with $L=-K_X\sim\{0\}+\{\infty\}$. Then $P=[-1,1]$. For $v\in\ZZ$, let $q=|v|+1$. Then
    $$
        P_{v,q}=\left\{\left(u,h\right)\in\left[-1,1\right]\times\RR_+\,:\,h\leq v\cdot u+q\right\}.
    $$
    This corresponds to $\overline{\cX}=\FF_{|v|}$ with $\overline{\cL}=E+H+qF_\infty$, where $E$ is the $-|v|$-curve, $H$ the torus invariant $|v|$-curve, and $F_\infty$ the fiber over $\infty\in\PP^1$. See Figure \ref{Fv example}.
    \begin{figure}
        \centering
        \begin{tikzpicture}
            \draw[<->](-1,0)node[left]{$H$}--(1,0)node[right]{$E$};
            \draw[<->](0,1)node[above]{$F_0$}--(0,0)--(1,-1)node[below right]{$F_\infty$};
            \draw(3,-1)--(5,-1)node[midway,below]{$F_0$}--(5,2)node[midway,right]{$H$}--(3,0)node[midway,above left]{$F_\infty$}--cycle node[midway,left]{$E$};
            \draw(7,-1)rectangle(9,2)node[midway]{$\overline{\cX}$};
            \draw[->](9.5,.5)--(10.5,.5)node[midway,above]{$\pi$};
            \draw(11,-1)node[right]{$0$}--(11,2)node[midway,right]{$\PP^1$}node[right]{$\infty$};
        \end{tikzpicture}
        \caption{The fan and the polytope $P_{v,q}$ of the pair $(\overline{\cX},\overline{\cL})$.}
        \label{Fv example}
    \end{figure}
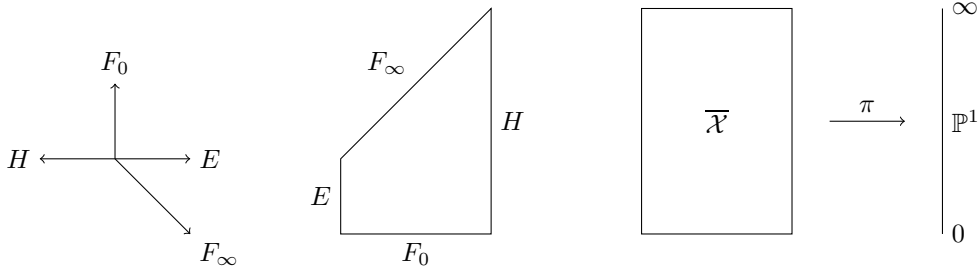
\end{example}

\begin{example}
    Let $X=\PP^2$ with $L=-K_X$. We have $v_1=(1,0)$, $v_2=(0,1)$, $v_3=(-1,-1)$. Let us verify Corollary \ref{combRR cor} with $v=(1,0)$.

    Let $D_1,D_2,D_3$ be the corresponding toric divisors. Geometrically, they correspond to virtual polytopes $P_{D_i}=P_{-K_X+D_i}-P_{-K_X}$. We have $L=-K_X=D_1+D_2+D_3$. By Corollary \ref{combRR cor},
    \begin{align*}
        a_2&=\frac{L^2}{2}=\frac{9}{2},\\
        a_1&=\frac{L^2}{2}=\frac{9}{2},\\
        a_0&=\frac{L^2+\sum\limits_{i<j}D_i\cdot D_j}{12}=1.
    \end{align*}
    For divisors in \eqref{cD def}, we have $\overline{\cL}'=-K_{\cX/\PP^1}=\cD_1+\cD_2+\cD_3$, and 
    \begin{align*}
        \cD_2&\sim\cD_3,&\cD_1^3&=-2,&\cD_1^2\cdot\cD_2&=-1,&\cD_1\cdot\cD_2^2&=0,&\cD_2^3&=1.
    \end{align*}
    By Corollary \ref{combRR cor},
    \begin{align*}
        c'_3&=\frac{{\overline{\cL}'}^3}{6}=0,\\
        c'_2&=\frac{{\overline{\cL}'}^3}{4}=0,\\
        c'_1&=\frac{\overline{\cL}'\cdot\left({\overline{\cL}'}^2+\sum\limits_{1\leq i<j\leq d}\cD_i\cdot\cD_j\right)}{12}=0.
    \end{align*}
    Therefore,
    $$
        \left\langle\Bc_k\left(P\right),v\right\rangle=\frac{\sum\limits_{j=0}^nc'_{j+1}k^j}{\sum\limits_{j=0}^na_jk^j}=0.
    $$
\end{example}

\subsection{Three non-stabilizing quantized barycenter examples}

\begin{example}
\lb{BSExam}
Batyrev--Selivanova symmetry of $P$ \cite{BS99} implies
$\Bc_k(P)=0$ for all $k\in\NN$, and is furthermore equivalent to it when $n\le6$ \cite{JLR24}. This relies on
    Nill--Paffenholz's exhaustive computer search showing there are exactly three toric Fano manifolds $X_1$, $X_2$, $X_3$ of dimension up to 8 whose associated polytopes have $\Bc=0$ but are not symmetric in the sense of Batyrev--Selivanova
    \cite[Proposition 2.1]{NP11}.

$X_1$ is 7-dimensional. Let $P_1$ be the associated polytope. Then a computer calculation shows
\begin{align}
\label{P1BCEq}
\Bc_1(P_1)&=\frac{16}{2257}(-1,-1,-1,1,1,1,2), \nonumber \\
\Bc_2(P_1)&=\frac{60}{27121}(-1,-1,-1,1,1,1,2), \nonumber \\
\Bc_3(P_1)&=\frac{2744}{2579721}(-1,-1,-1,1,1,1,2).
\end{align}
It is interesting to note that our results give a refinement of
a theorem of Ono--Sano--Yotsutani \cite[Theorem 1.6]{OSY12}, who showed 
that $(X_1,-kK_{X_1})$ is not Chow semistable for $k$ large enough.
Chow stability of $(X_1,-kK_{X_1})$ is equivalent to existence of $k$-anticanonical balanced metrics (see \cite[Theorem 4]{PS03}, \cite[Theorem 3.2]{Zh96}, \cite[Theorem 0.1]{Luo98}),
so this means Ono--Sano--Yotsutani showed that $X_1$ does not admit $k$-anticanonical balanced metrics
for $k$ large enough. But in fact from the above computation and Corollary \ref{Bc_k=0 for all k if large k}, $\delta_k(-K_{X_1})<1$ for all $k$ except at most $7$ values of $k$.
Thus, by \cite[Theorem 2.3]{RTZ21}, 
$(X_1,-kK_{X_1})$ is actually not Chow semistable for all but at most $7$ values of $k$.

Next, $X_2=X_1\times\PP^1$. So $P_2=P_1\times[-1,1]$. Then $\Bc_k(P_2)=(\Bc_k(P_1),0)$ (see \eqref{P1BCEq}).

Finally,
$X_3$ is 8-dimensional. Let $P_3$ be the associated polytope. Then a computer calculation shows
\begin{align*}
\Bc_1(P_3)&=\frac{32}{5459}(-1,-1,-1,1,1,1,1,2),\\
\Bc_2(P_3)&=\frac{580}{321787}(-1,-1,-1,1,1,1,1,2).
\end{align*}
Note that in these three examples, the $\Bc_k(P_i)$ we computed
are colinear. This is actually true for all $k$ and is because
while $P_i$ is not Batyrev--Selivanova symmetric, there is a 1-dimensional
subspace that is invariant under $\Aut P$. We explain this in more detail in \cite{JLR24}. 
\end{example}
\appendix
\section{On a result of Blum--Jonsson and Rubinstein--Tian--Zhang}

\bigskip

{\bf\large Yaxiong Liu\footnote{University of Maryland, yxliu238@umd.edu}}

\bigskip

Blum--Jonsson computed explicitly the $\delta$-invariant of toric 
varieties \cite[Corollary 7.16]{BJ20} using a valuative approach.
In the Fano case, Rubinstein--Tian--Zhang gave an alternative proof of their result 
and also extended it to the $\delta_k$-invariants
by using an analytic approach with balanced metrics and coercivity
arguments \cite[Corollary 7.1]{RTZ21}.
The goal of this appendix is to give an alternative proof of the
Rubinstein--Tian--Zhang result and simultaneously extend it to any polarization, by showing that (as expected) the original valuative proof
of Blum--Jonsson for $\delta$ adapts to handle also $\delta_k$.

We follow the valuative reformulation of the $\delta$-invariant due to Blum--Jonsson.
Let $(X,L)$ be a polarized variety. 
A \textit{valuation} on $X$ means a real valuation $w:\CC(X)^*\rightarrow\RR$ that is trivial on $\CC$. We denote $\Val_X$ the space of all valuations on $X$.
A \textit{prime divisor} $E$ \textit{over} $X$, meaning that $E$ is a prime divisor on some birational model $Y$ of $X$, defines a valuation $\ord_E:\CC(X)^*\rightarrow\RR$ in $\Val_X$ as vanishing order at the generic point of $E$. 
For any prime divisor $E$ over $X$, the \textit{log discrepancy} of $\ord_E$ is defined by $A_X(\ord_E):=1+\ord_E(K_{Y/X})$.
We assume that $X$ has klt singularities.
As showed in \cite{BFFU11}, the log discrepancy can be extended to a lower semicontinuous function $A_X:\Val_X\rightarrow[0,+\infty]$ that is homogeneous of order $1$, i.e., $A_X(tw)=tA_X(w)$ for $t\in\RR_+$.

\begin{definition}
\label{def filtration}
A \textit{filtration} $\mathcal{F}R_{\bullet}$ of the graded $\mathbb{C}$-algebra $R=\oplus^{\infty}_{k=0}R_k=\oplus^{\infty}_{k=0}H^0(X,kL)$ consists of a family of subspace $\{\mathcal{F}^{\lambda}R_k \}_{\lambda}$ of $R_k$, for $\lambda\in\mathbb{R}$ and $k\in\mathbb{Z}_{+}$, satisfying:
		\begin{enumerate}[(i)]
			\item (decreasing) $\mathcal{F}^{\lambda}R_k \subset\mathcal{F}^{\lambda^{\prime}}R_k$ if $\lambda\geq \lambda^{\prime}$;
			\item (left-continuous) $\mathcal{F}^{\lambda}R_k=\cap_{\lambda^{\prime}<\lambda} \mathcal{F}^{\lambda^{\prime}}R_k$;
			\item (multiplicative) $\mathcal{F}^{\lambda}R_k\cdot\mathcal{F}^{\lambda^{\prime}}R_{k^{\prime}}
			\subset\mathcal{F}^{\lambda+\lambda^{\prime}}R_{k+k^{\prime}}$
			for any $\lambda,\lambda^{\prime}\in\mathbb{R}$ and $k,k^{\prime}\in\mathbb{Z}_{+}$;
			\item (linearly bounded)
            $\mathcal{F}^{0}R_k=R_k$ and
			 $\mathcal{F}^{\lambda}R_k=0$ for $\lambda>>0$.
          \end{enumerate}
\end{definition}

Any valuation $w\in\Val_X$ induces a filtration $\cF_w$ on the section ring $R(X,L)$ via
\begin{equation*}
    \cF_w^\lambda R_k:=\{s\in R_k~:~v(s)\geq \lambda\}
\end{equation*}
for $k\in\NN$ and $\lambda\in\RR_+$, where $R_k=H^0(X,kL)$.
The jumping numbers of the filtration $\cF_w$ are
\begin{equation}
\label{def jumping number}
    a_{k,j}:=\inf\{t\in\RR_+~:~\codim\cF^tR_k\geq j\},
    \quad 1\leq j\leq N_k.
\end{equation}
Their average is 
\begin{equation}
\label{S_k def}
    S_k(w):=\frac{1}{kN_k}\sum_{j} a_{k,j}
\end{equation}
where $N_k:=h^0(X,kL)$.
The \textit{expected vanishing order} is 
\begin{equation*}
    S(w):=\frac{1}{\Vol(L)}\int_0^\infty \Vol(L;w\geq \lambda)d\lambda
\end{equation*}
where
$    \Vol(L;w\geq \lambda):=\lim_{k\rightarrow\infty}\frac{n!}{k^n}\dim\cF_w^{\lambda k}R_k,$
and  
$
    S(w)=\lim_{k\rightarrow\infty}S_k(w)
$
\cite[Theorem 2.11]{BC11}.

The $\delta_k$- and $\delta$-invariants can be expressed using valuations:
\begin{equation*}
    \delta_k(L)=\inf_w\frac{A_X(w)}{S_k(w)}, \quad
    \delta(L)=\lim_{k\rightarrow\infty}\delta_k(L)=\inf_w\frac{A_X(w)}{S(w)},
\end{equation*}
where $w$ ranges over nontrivial valuations on $X$ with $A_X(w)<\infty$
\cite[Proposition 4.2, Theorem~4.4]{BJ20}.

Blum--Jonsson's proved that for polarized toric variety $(X,L)$ \cite[Corollary 7.16]{BJ20}, 
\begin{align}
\label{toric delta}
	\delta(L)
	=\inf_{w\in {\Val_X}}\frac{A_X(w)}{S(w)}
        =\inf_{v\in N_\RR\setminus\{0\}}\frac{A_X(v)}{S(v)}
	=\min_{i=1,\ldots,d}\frac{1}{\langle\Bc(P),v_i\rangle+b_i}.
\end{align}
They used the following proposition, reducing the above formula to the toric valuation $v\in N_\RR\setminus\{0\}$:
\begin{proposition}
\label{Prop of BJ}
If $w\in \Val_X$ with $A_X(w)<\infty$, then there exists $v\in N_\RR\backslash \{0\}$ such that
	$
	A(v)\leq A(w)$ and  $S(v)\geq S(w)$.
\end{proposition}
We now explain how to modify Blum--Jonsson's argument to derive a similar conclusion for $S_k$ 
(Proposition \ref{modified prop}).
For $\lambda\in\RR_+$ and $k\in \NN$,
\begin{align*}
	\ssb_{\lambda,k}(\cF)
	:=\kb(|\cF^\lambda R_k|),
\end{align*} 
where $\kb(|\cF^\lambda R_k|)\subset \cO_X$ is the \textit{base ideal} of the linear series $|\cF^\lambda R_k|$  \cite[Definition 1.1.8]{Laz04} defined as the image of the map $|\cF^\lambda R_k|\otimes_\CC L^*\rightarrow\cO_X$ determined by the evaluation of sections in $|\cF^\lambda R_k|$ 
\begin{align*}
    \cF^\lambda R_k\otimes_\CC\cO_X\rightarrow L.
\end{align*}

The sequence $\ssb_{\lambda,k}(\cF)$ is stationary, i.e., $\ssb_{\lambda,k}(\cF)=\ssb_{\lambda,k+1}(\cF)$ for $k$ large enough \cite[Lemma 3.17]{BJ20}, thus one set $\ssb_\lambda(\cF):=\ssb_{\lambda,k}(\cF)$ for large $k$.
One has \cite[Corollary 3.18]{BJ20},
\begin{align*}
	\ssb_\lambda(\cF)\cdot\ssb_{\lambda'}(\cF)
	\subset \ssb_{\lambda+\lambda'}(\cF),
\end{align*}
so $	w(\ssb_{\lambda+\lambda'}(\cF))
	\leq w(\ssb_\lambda(\cF)) +w(\ssb_{\lambda'}(\cF))
$ for any $\lambda,\lambda'\in\RR_+$ and any $w\in\Val_X$.

Recall the classical Fekete Lemma:
\begin{lemma}
{\rm\cite[Section 2]{Fek23}, \cite[Theorem 7.6.1]{HP48}}
\label{Fekete Lem}
    For any measurable subadditive function $f:(0,\infty)\rightarrow\RR$, the limit $\lim_{\lambda\rightarrow\infty}f(\lambda)/\lambda$ exists and is equal to $\inf_{\lambda>0}f(\lambda)/\lambda$ $($This limit may be $-\infty$$)$.
\end{lemma}

\begin{definition}
\label{def of valu of graded ideal}
    The valuation of the graded sequence of base ideals $\ssb_\bullet(\cF)$ of a filtration is defined as
\begin{equation}
\label{valuation for real index}
	w(\ssb_\bullet(\cF))
	:=\lim_{\lambda\rightarrow\infty}\frac{w(\ssb_\lambda(\cF))}{\lambda}
	=\inf_{\lambda\in\RR_+}\frac{w(\ssb_\lambda(\cF))}{\lambda},
\end{equation}
where $w\in\Val_X$.
Due to Lemma \ref{Fekete Lem}, it is well-defined.
\end{definition}

\begin{lemma}
\label{inclusion for all real}
	If $w\in\Val_X$ with $w(\ssb_\bullet(\cF))\geq1$, then $\cF^\lambda R_k\subset \cF_w^\lambda R_k$ for any $\lambda\in\RR_+$ and all $k\in\NN$.
\end{lemma}
In \cite[Lemma 3.20]{BJ20}, the authors showed the inclusion for all $\lambda\in\NN$. By modifying the definition of the valuation of the graded base ideals, we also obtain the same conclusion for all $\lambda\in\RR_+$.
\begin{proof}
	By definition (\ref{valuation for real index}), we have
	$1\leq w(\ssb_\bullet(\cF))\leq w(\ssb_\lambda(\cF))/\lambda$ for all $\lambda>0$.
	Then $w(\ssb_\lambda(\cF))\geq\lambda$, so that $\ssb_\lambda(\cF)\subset\ssa_\lambda(w)$ where
$		\ssa_\lambda(w):=\{f\in\cO_X~|~w(f)\geq\lambda \}$.
	Since we have $\ssb_{\lambda,k}(\cF)\subset \ssb_\lambda(\cF)$ \cite[Lemma 3.17]{BJ20}, 
	$	\cF^\lambda R_k
		\subset H^0(X,kL\otimes\ssb_{\lambda,k}(\cF))
		\subset H^0(X,kL\otimes\ka_\lambda(w))
		=\cF_w^\lambda R_k.
	$
 \end{proof}
Next, we adapt \cite[Corollary 3.21]{BJ20}:
\begin{corollary}
\label{modified cor}
	Let $\cF$ be a linearly bounded filtration of $R(X,L)$. Then,
 	for all $w\in\Val_X, k\in\NN$,
	\begin{align*}
		S_k(w)\geq w(\ssb_\bullet(\cF))S_k(\cF).
	\end{align*}
\end{corollary}
\begin{proof}
	If $w(\ssb_\bullet(\cF))=0$, there is nothing to show.
	If not, we may assume $w(\ssb_\bullet(\cF))=1$ by scaling $w$.
	By Lemma \ref{inclusion for all real}, we have $\cF^\lambda R_k\subset \cF_w^\lambda R_k$ for all $\lambda>0$. Thus, the jumping numbers of the filtration $\cF_w^{\bullet} R_k$ is greater than equal to the jumping numbers of the filtration $\cF^\bullet R_k$ (see (\ref{def jumping number})). 
    Since $S_k(\cF)$ is the sum of jumping numbers of the filtration $\cF$, we obtain
	$
		S_k(\cF)\leq S_k(\cF_w)=S_k(w).
	$ 
\end{proof}

Once we have Corollary \ref{modified cor}, then we can run the proof of \cite[Proposition 7.14]{BJ20} and obtain:
\begin{proposition}
\label{modified prop}
    For a nontrivial $w\in\Val_X$ with $A_X(w)<\infty$, then there exists $v\in\NN_\RR\backslash\{0\}$ with 
$
	A(v)\leq A(w)
 $
 and
 $S_k(v)\geq S_k(w)$
   \hbox{ for all $k\in\NN$}.

\end{proposition}
Thus, we obtain the following extension of the result of 
Rubinstein--Tian--Zhang (that assumed $L=-K_X$ so $b_i=1$ for all $i$):
\begin{corollary}
\label{toric delta_k}
For a polarized toric variety $(X,L)$ and $k\in\NN$,
\begin{align}
\label{toric delta_k formula}
	\delta_k(L)
       =\min_{i=1,\ldots,d}\frac{1}{\langle \Bc_k(P),v_i\rangle+b_i}.
\end{align}
\end{corollary}

    Zhang \cite{Zh18} deduced an explicit formula of $\delta$-invariant in the rational Fano $T$-varieties of complexity $1$, generalizing Blum--Jonsson's result.
    By the same ideas as above, we can also modify Zhang's argument to obtain a similar formula, involving the weighted quantized barycenter of its associated polytope, of $\delta_k$-invariant for $T$-varieties of complexity $1$.

\begin{proof}
    By Proposition \ref{modified prop},
\begin{align}
\label{delta_k reduced to toric}
	\delta_k(L)=\inf_{v\in N_\RR\backslash\{0\}}\frac{A_X(v)}{S_k(v)}.
\end{align}
Recall the following facts due to Blum--Jonsson \cite{BJ20}:
\begin{enumerate}[(i)]
	\item $S_k(v)=\langle\Bc_k(P),v\rangle-\psi_L(v)$,
	where $\psi_L$ is the support function of $P$ \cite[Proposition 7.6]{BJ20}.
	\item  Write $L=\cO_X(D)$ and $D_u:=D+\sum_{i=1}^d\langle u,v_i\rangle D_i$ for $u\in P\cap M_\QQ$. Then $v(D_u)=\langle u,v\rangle-\psi_L(v)$ for $v\in N_\RR$ \cite[Lemma 7.1]{BJ20}.
	\item\label{BJ cor 7,3}  For $u\in P\cap M_\QQ$, 
 $		\lct(D_u)=\inf_{v\in N_\RR\backslash\{0\}}\frac{A_X(v)}{v(D_u)}
		=\min_{i=1,\ldots,d}\frac{1}{\langle u,v_i\rangle+b_i}
	$ \cite[Corollary 7.3]{BJ20}.
\end{enumerate}
Thus,
\begin{align*}	\delta_k(L)=&\inf_{v\in N_\RR\backslash\{0\}}\frac{A_X(v)}{S_k(v)} \text{\ \ \ \  (by (\ref{delta_k reduced to toric}))} \\
	=&\inf_{v\in N_\RR\backslash\{0\}}\frac{A_X(v)}{\langle\Bc_k(P),v\rangle-\psi_L(v)} \text{\ \ \ \ (by (i))}  \\
	=&\inf_{v\in N_\RR\backslash\{0\}}\frac{A_X(v)}{v(D_{\Bc_k(P)})} \text{\ \ \ \  (by (ii))}  \\
	=&\lct(D_{\Bc_k(P)}) \text{\ \ \ \  (by (iii))} \\
	=&\min_{i=1,\ldots,d}\frac{1}{\langle \Bc_k(P),v_i\rangle+b_i}
	\text{\ \ \ \  (by (iii))}
\end{align*}
as claimed.
\end{proof}

\bigskip
\textsc{University of Maryland}

\bigskip
{\tt cjin123@terpmail.umd.edu, yanir@alum.mit.edu}

\end{document}